\documentclass[12pt]{amsart}
\usepackage{amsmath,amsfonts,amssymb,amsthm,amscd}
\usepackage{pgfplots}
\numberwithin{equation}{section}
\theoremstyle{plain}
\newtheorem{theorem}{Theorem}
\newtheorem{lemma}{Lemma}[section]
\newtheorem{propos}{Proposition}
\newtheorem{corollary}{Corollary}
\theoremstyle{definition} 
\newtheorem{definition}{Definition}
\newtheorem{remark}{Remark}
\newtheorem{example}{Example}
\begin{document}
\begin{flushright}  Dedicated to Victor Matveevich Buchstaber on the occasion of his 75th birthday

 \end{flushright}
\title{Geometry of central extensions of nilpotent Lie algebras}
\author[D.\,V.~Millionshchikov]{D.\,V.~Millionshchikov}
\address{Department of Mechanics and Mathematics, Moscow State University, 119992 Moscow, Russia}
\email{mitia\_m@hotmail.com}
\author[ R.~Jimenez]{R.~Jimenez}
\address{National Autonomous University of Mexico }
\email{landojb1960@gmail.com}

\date{21.01.2019}

\maketitle

\begin{abstract}
We obtain a recurrent and monotone method for constructing and classifying nilpotent Lie algebras by means of successive central extensions in this paper. It consists in calculating the second cohomology $H^2({\mathfrak g}, {\mathbb K})$ of an extendable nilpotent Lie algebra ${\mathfrak g}$ with the subsequent study of the orbit space geometry of the automorphism group ${\rm Aut}({\mathfrak g})$ action on Grassmannians of the form $Gr(m,H^2({\mathfrak g}, {\mathbb K}))$.  In this case, it is necessary to take into account the filtered cohomology structure with respect to the ideals of the lower central series: a cocycle defining a central extension must have maximum filtration. Such a geometric method allows us to classify nilpotent Lie algebras of small dimensions, as well as to classify narrow naturally graded Lie algebras. The concept of a rigid central extension is introduced. Examples of rigid and non-rigid central extensions are constructed.
\end{abstract}


\markright{Geometry of central extensions}

\footnotetext[0]{The study of the first author was suppored by RFBR grant N 16-51-55017, the second author was partially supported by organizations PASPA-DGAPA, UNAM and
CONACyT.}

\section*{Introduction}\label{s0}
An arbitrary nilpotent Lie algebra is a central extension of a nilpotent Lie algebra of lower dimension.
Question: Is it possible to organize a recurrent procedure using such a construction and classify finite-dimensional nilpotent Lie algebras?

The very first analysis of the posed question  shows its transcendental complexity, the answer to it is hardly accesible in a general setting and for an arbitrary dimension, but in small dimensions or for some special classes of nilpotent Lie algebras, answers can still be obtained.

We start the study with small dimensions. According to Morozov's well-known classification \cite{Mor}, in dimensions $\le 6$ there exists a finite number of pairwise non-isomorphic nilpotet Lie algebras over a field of characteristic zero. Starting with the dimension $7$ (where a one-parameter family of pairwise non-isomorphic nilpotent Lie algebras appears), the difficulties of classifying nilpotent Lie algebras are rapidly increasing, which leads, in particular, to the need to consider the so-called affine ${\mathcal L}_n$ variety of Lie algebra structures on a fixed $n$-dimensional vector space $V$ over the field  ${\mathbb K}$.
 The manifold ${\mathcal L}_n$ consists of skew-symmetric bilinear mappings
$\mu: V \wedge V \to V$ satisfying the Jacobi identity.
The affine variety ${\mathcal N}_n $ of nilpotent Lie algebras is also defined. There is a natural $ GL_n$-action on $ {\mathcal L}_n$ (respectively on $ {\mathcal N}_n$: 
$$
(g \cdot \mu)(x,y)=g\left( \mu(g^{-1}x, g^{-1}y)\right),   \; g \in GL_n, \; x,y \in V.
$$
Obviously, the isomorphism class of a given algebra (structure) of Lie $\mu \in {\mathcal L}_n $ corresponds to the orbit $O(\mu)$ of this action.

Here it is worth making a digression and say that the study of the variety ${\mathcal L}_n$ from the point of view of the orbit space of the action of the full linear group $GL_n$ has long attracted the attention of algebraists \cite{CaDu, KirNer, Ner, BuSt, FiPe1, Gorb2, Mnvl, FiPe}, as well as the study  of the variety ${\mathcal N}_n$  of nilpotent Lie algebras \cite{V, GrOH1, GJMKh, GoRemm}. This task is almost trivial for $n = 2$, but it is no longer such, starting with $n=3$ \cite{KirNer}, and for $n = 4$ the results of such studies look very, very complicated \cite{FiPe1, Mnvl}.  Note also that
the study of the orbits of actions of algebraic groups on affine varieties is a classic subject of algebraic geometry and the theory of invariants.
But such a theory is well developed for objects defined over an algebraically closed field, and we will be interested, first of all, in geometric applications for the field of real numbers. We also note that the issues considered in this article are directly related to the theory of deformations of Lie algebras (see \cite{Fial2}). 
Questions related to the applications of cohomological calculations to the explicit construction of formal deformations of a Lie algebra were considered in \cite{Fial3}. Note that the formal deformation technique developed in \cite{Fial3} is especially useful in a finite-dimensional situation.

A Lie algebra $\mu$ is called rigid if its orbit $O(\mu) $ is open. We will adhere to the principle of parallel consideration of the Euclidean topology of a finite-dimensional space together with the Zariski topology when studying the action orbits \cite{Gorb}.
It is this visual geometric approach that will be the basis of all our research.

The topic of this artcile was influences seriously by Vergne's conjecture of $1970$
{\it in the variety $N_n$ of nilpotent Lie algebra of dimension $n \ge 8$ there are no rigid Lie algebras}. The later and still open conjecture by Grunewald and O'Halloran that {\it each nilpotent Lie algebra of dimension at least two is a contraction (degeneration) of some other Lie algebra \cite{GrOH}} is logically
adjacent to this conjecture.
Recall that a Lie algebra $\mu'$ is called a contraction (degeneration) of $\mu$, if
$\mu' \in \bar O(\mu)$, where $\bar O(\mu)$ stands for the Zarissky closure of the orbit $O(\mu)$. Recently it was proved in \cite{HGrT} that the Gr\"unewald-O'Halloran conjecture is true for  $7$-dimensional nilpotent Lie algebras. In general, both conjectures are still open.

We briefly describe the structure of this work. In the section \ref{s1} we give all the necessary definitions and information concerning nilpotent Lie algebras, including finite-dimensional positively graded and filtered Lie algebras.
In the section \ref{s2} we give all the necessary information about the central extensions of an arbitrary Lie algebra ${\mathfrak g} $. General facts about the cohomology of finite-dimensional positively graded and filtered Lie algebras
are contained in \ref{s3}.
The key point in this article is the section \ref{s4}, which describes the recurrent and monotone method for constructing and classifying finite-dimensional nilpotent Lie algebras. The inductive step of our recurrent procedure is to construct a central extension $\tilde {\mathfrak g} $ of a nilpotent Lie algebra ${\mathfrak g} $, which has a nil-index $s$ one more than ${\mathfrak g}$ has and it has the given dimension of the $s$th ideal $\tilde {\mathfrak g}^s$ of the lower central series of the Lie algebra $\tilde{\mathfrak g}$.
Such an extension does not exist for every nilpotent Lie algebra. All information about extensions with the required properties is contained in the filtered structure.
the spaces $H^2 ({\mathfrak g}, {\mathbb K})$, whose filtration is induced by the natural filtration of the Lie algebra ${\mathfrak g} $ by the ideals of its lower central series.
To classify up to an isomorphism of $m$-dimensional extensions of the nilpotent Lie algebra $ {\mathfrak g}$, one must study the orbit space of the group action
$GL_m \times {\rm Aut}({\mathfrak g})$ on the second cohomology Grassmannian
$Gr(m, H^2({\mathfrak g}, {\mathbb K}))$. The open orbits of such an action will be called rigid cocycles (rigid sets of cocycles), and the corresponding central extensions will be called rigid central extensions in the section \ref{s5}. 
More accurately:
we say that the central extension $\mathfrak {g}_{\alpha}$ defined by the cocycle $ \alpha \in
H^2(\mathfrak {g}, {\mathbb K}^m)$ is rigid if an arbitrary central extension
$\mathfrak{g}_{\beta}$ defined by the cocycle $ \beta $ close to $ \alpha $ (in the sense of the Euclidean topology of the space $H^2(\mathfrak {g}, {\mathbb K}^m), {\mathbb K} = {\mathbb R}, {\mathbb C}$) is isomorphic to $\mathfrak{g}_{\alpha}$. On a nilpotent Lie algebra of dimension $\le 5$
there is always a one-dimensional rigid central extension. This is a simple consequence of Morozov’s classification \cite{Mor}.
Section \ref{s5} gives an example \ref{m_2} of a nilpotent algebra of dimension $6$ that does not admit any rigid one-dimensional central extension.
It should be noted that we do not endow the orbit space of the action $GL_m \times {\rm Aut} ({\mathfrak g}) $ on the Grassmannian $Gr(H^2({\mathfrak g}, {\mathbb K}), m)$ neither the structure of a topological space, nor, all the more, an algebraic set.

Section \ref{s6} is devoted to the classification of naturally graded Lie algebras, i.e. such nilpotent Lie algebras ${\mathfrak g} $ that are isomorphic to their associated graded Lie algebra ${\rm gr}{\mathfrak g} $ with respect to filtration by ideals of the lower central series.
As an example demonstrating the possibilities of our method of successive central extensions, we present a new proof of the Vergne theorem \cite{V}  on the classification of naturally graded filiform Lie algebras.

In the last section \ref{s7}, we construct an important example of a $6$-dimensional naturally graded Lie algebra $\tilde{\mathcal L}(2,4)$, which also has no rigid central extensions. But the main goal of the construction of this example is to show that dropping the "$3/2$ width" condition from \cite{Mill5} leads to the appearance of parametric families of pairwise nonisomorphic naturally graded Lie algebras already in dimension $8$. Of particular interest is the explicit visual description of the orbits of the action of the subgroup ${\rm Aut} _{gr}({\mathfrak g})$ graded automorphisms of the Lie algebra ${\mathfrak g} $ in the form of second-order surfaces in three-dimensional space.

\section{Nilpotent, positively graded and filtered Lie algebras}\label{s1}

A sequence of ideals of Lie algebra $\mathfrak{g}$
$$
\mathfrak{g}^1=\mathfrak{g} \; \supset \;
\mathfrak{g}^2=[\mathfrak{g},\mathfrak{g}] \; \supset \; \dots
\; \supset \;
\mathfrak{g}^k=[\mathfrak{g},\mathfrak{g}^{k-1}] \; \supset
\; \dots
$$
The sequence of ideals of a Lie algebra is called a decreasing (lower) central series of a Lie algebra $\mathfrak{g}$.

A Lie algebra $\mathfrak{g}$ is called nilpotent if
there is a natural number $s$ such that
$$
\mathfrak{g}^{s+1}=[\mathfrak{g},\mathfrak{g}^s]=0,
\quad \mathfrak{g}^s \: \ne 0.
$$
The number $s$ is called the nil-index of a nilpotent Lie algebra
$\mathfrak{g}$, 
and Lie algebra itself
$\mathfrak {g}$ is called a nilpotent Lie algebra of index $s$ or a nilpotent Lie algebra of degree $s$.

\begin{definition}
The Lie algebra ${\mathfrak g}=\oplus_{i=1}^{+\infty}{\mathfrak g}_i$, decomposed into a direct sum of its homogeneous subspaces ${\mathfrak g}_i, i \in {\mathbb N}, $ is called ${\mathbb N}$-graded (positively graded) if the following condition holds
$$
[{\mathfrak g}_i, {\mathfrak g}_j]\subset {\mathfrak g}_{i+j}, i,j \in {\mathbb N}.
$$
\end{definition}
A finite-dimensional ${\mathbb N}$-graded Lie algebra is nilpotent.

In the definitions of Lie algebras, we will omit the relations of the form
$[e_i, e_j]=0$.
\begin{example}
\label{m_0(n)}
Lie algebra $\mathfrak{m}_0(n-1)$, defined by its base
$ e_1, e_2, \dots, e_n$ with commutation relations
$$ 
[e_1,e_i]=e_{i+1}, \; 2\le i \le n{-}1,
$$
is nilpotent with nil-index $s(\mathfrak{m}_0(n-1))=n-1$. 
The Lie algebra $\mathfrak{m}_0(n-1)$ can be equipped with a grading
$\mathfrak{m}_0(n-1)=\oplus_{i=1}^{n}(\mathfrak{m}_0(n-1))_i$, where
all homogeneous subspaces $(\mathfrak{m}_0(n-1))_i=\langle e_i \rangle$ are one-dimensional for  $1 \le i \le n$.
\end{example}
\begin{propos}[\cite{V}]
Let $\mathfrak{g}$ be a $n$-dimensional nilpotent Lie algebra.
Then for its nil-index $s(\mathfrak {g}) $ the estimate 
 $s(\mathfrak{g}) \le n-1$ is true.
\end{propos}
\begin{definition}
A nilpotent $n$-dimensional Lie algebra $\mathfrak{g}$ is called filiform if $s(\mathfrak{g})=n-1$.
\end{definition}

The Lie algebra ${\mathfrak m}_0 (n-1)$, considered above, is filiform. We give another example of a filiform Lie algebra.
\begin{example}
\label{m_1(2m-1)}
The Lie algebra $\mathfrak{m}_1(2m{-}1), m \ge 3$.  The basis $e_1,e_2,\dots,e_{2m{-}1},e_{2m}$ and structure relations
$$
\begin{array}{c}
[e_1,e_i]=e_{i+1}, i=2,\dots, 2m{-}2,\;
[e_{k},e_{2m-k+1}]=(-1)^{k+1}e_{2m}, k=2,\dots,m.
\end{array}
$$
\end{example}

\begin{definition}
A set of nested subspaces of a Lie algebra  ${\mathfrak g}$  
$$
F^1{\mathfrak g}={\mathfrak g} \supset F^2{\mathfrak g} \supset F^3{\mathfrak g}\supset \dots \supset F^m{\mathfrak g} \supset \dots
$$
is called positive filtration, if
$$
[F^i{\mathfrak g}, F^j{\mathfrak g}] \subset F^{i+j}{\mathfrak g}, \forall i,j \in {\mathbb N}.
$$
\end{definition}
An example of positive filtration is the filtration by the ideals ${\mathfrak g}^k$ of the lower central series of a nilpotent Lie algebra ${\mathfrak g}$.

\begin{propos}
A Lie algebra ${\mathfrak g}$ with a positive filtration $F^1{\mathfrak g} \supset \dots \supset F^m{\mathfrak g} \supset \{0\}$ of finite length $m$ is nilpotent.
\end{propos}
\begin{proof}
It is obvious that
${\mathfrak g}^2=[{\mathfrak g},{\mathfrak g}]=[F^1{\mathfrak g}, F^1{\mathfrak g}] \subset F^2{\mathfrak g}$. 
Continuing by induction, we get the inclusion
${\mathfrak g}^k  \subset F^k{\mathfrak g}$ for all  $k \ge 2.$
\end{proof}

One can define for an arbitrary  ${\mathbb N}$-graded Lie algebra  ${\mathfrak g}=\oplus_{k=1}{\mathfrak g}_k$ th filtration related to its grading
$$
F^k{\mathfrak g}=\oplus_{i=k}^{+\infty}{\mathfrak g}_i, \; k \ge 1.
$$

For a filtered Lie algebra $({\mathfrak g},F)$, the associated graded Lie algebra is defined
${\rm gr}_F \,\mathfrak{g}=\oplus_{i=1}^{+\infty}  F^i\mathfrak{g} / F^{i+1}\mathfrak{g}$. Its Lie bracket is given by the formula
$$
 \left[x+F^{i+1}\mathfrak{g}, y+F^{j+1}\mathfrak{g}\right]=[x,y]+F^{i+j+1}\mathfrak{g}, x\in F^i\mathfrak{g}, y\in F^j\mathfrak{g}.
$$
\begin{definition}
A nilpotent Lie algebra $\mathfrak {g}$ is called a naturally graded (Carnot algebra) if it is isomorphic to its associated graded Lie algebra
${\rm gr} \, \mathfrak {g}$ with respect to filtration by ideals of the lower central series.
The grading $\mathfrak{g}=\oplus_{i=1}^{+\infty}{\mathfrak g}_i$ of a naturally graded Lie algebra $\mathfrak{g}$ is called a natural grading if there is a graded isomorphism
$$
\varphi: {\rm gr}\, \mathfrak{g} \to \mathfrak{g}, \; \varphi(({\rm gr} \,\mathfrak{g})_i)={\mathfrak g}_i, \; 
i \in {\mathbb N}.
$$
\end{definition}
In the sequel, by a naturally graded Lie algebra (Carnot algebra) ${\mathfrak g} = \oplus_{i=1}^{+\infty}{\mathfrak g}_i $, we will mean a naturally graded Lie algebra equipped with a natural grading.

The Lie algebra $\mathfrak{m}_0(n)$, considered above, is a naturally graded (Carnot algebra). However, its natural graduation differs from its graduation, which we considered at the very beginning. In particular, its very first homogeneous subspace
$\left({\rm gr} \, m_0(n)\right)_1$ is two-dimensional.
$$
\left( {\rm gr} \,m_0(n)\right)_1=\langle e_1, e_2\rangle, 
\left( {\rm gr}\, m_0(n)\right)_i=\langle e_{i{+}1}\rangle, i=2,\dots,n{-}1.
$$

From the properties of a decreasing (lower) central series, one can derive one very important property of natural grading  ${\mathfrak g}=\oplus_{i=1}^{+\infty}{\mathfrak g}_i$.

\begin{propos}
A positive grading  ${\mathfrak g}=\oplus_{i=1}^{+\infty}{\mathfrak g}_i$ of a Lie algebra ${\mathfrak g}$ is natural if and only if
$$
[{\mathfrak g}_1,{\mathfrak g}_i]={\mathfrak g}_{i{+}1}, i \ge 1.
$$
\end{propos}
In particular a naturally graded Lie algebra  ${\mathfrak g}=\oplus_{i=1}^{+\infty}{\mathfrak g}_i$ is generated by its first homogeneous component ${\mathfrak g}_1$.

Nilpotent Lie algebras of nil-index two $s({\mathfrak g})=2$ are also known as metabelian Lie algebras. The lower central series of metabelian Lie algebras is as simple as possible.
$$
\mathfrak{g}=\mathfrak{g}^1 \supset 
\mathfrak{g}^2=[\mathfrak{g},\mathfrak{g}] \supset 0.
$$
Obviously, an arbitrary metabelian Lie algebra is a naturally graded Lie (Carnot algebra). Indeed, choose a linear subspace ${\mathfrak g}_1$ in the Lie algebra $ {\mathfrak g} $ as an addition to the commutator $\mathfrak{g}_2=[\mathfrak{g},\mathfrak{g}]$ in the Lie algebra ${\mathfrak g}$: $\mathfrak{g}={\mathfrak g}_1\oplus \mathfrak{g}_2$. Obviously, with such a choice of subspaces ${\mathfrak g}_i, i = 1,2$,
$$
[{\mathfrak g}_1, {\mathfrak g}_1]=[\mathfrak{g},\mathfrak{g}]={\mathfrak g}_2.
$$
Classification of metabelian Lie algebras turned out to be extremely complicated \cite {Gaug, GalTim}: the maximum dimension, in which complex metabelian Lie algebras are currently classified, is $9$ \cite{GalTim}. 

But filiform naturally naturally graded Lie algebras are easy to classify \cite{V}. We will talk about such a classification in paragraph \ref{s6}.

\section{Central extensions of Lie algebras}\label{s2}

\begin{definition}
Central extension of Lie algebra
$\mathfrak{g}$ is called the exact sequence
\begin{equation}
\label{exactseq}
\begin{CD}0 @>>> V@>{i}>>\tilde {\mathfrak g} @>\pi>>{\mathfrak g}@>>>0
\end{CD}
\end{equation}
Lie algebras and their homomorphisms, in which the image of the homomorphism
$i: V \to \tilde {\mathfrak{g}}$ 
is contained in the center $Z(\tilde{\mathfrak{g}})$ of the Lie algebra $\tilde{\mathfrak{g}}$, and the linear subspace $V$ is considered as an abelian Lie algebra.
\end{definition}

\begin{example}
The Lie algebras $\mathfrak{m}_1(2m-1)$ and $\mathfrak{m}_0(2m-1)$   are one-dimensional central extensions of the Lie algebra 
$\mathfrak{m}_0(2m-2)$ for $m \ge 3$.
\end{example}

As a vector space, the central extension $\tilde {\mathfrak {g}}$ is a direct sum
$V \oplus \mathfrak {g} $ with standard inclusion $i$ and projection
 $\pi$.
The Lie bracket in the vector space $V \oplus \mathfrak {g}$ can be defined by the formula
$$
[(v,g), (w, h)]_{\tilde {\mathfrak g}}=(c(g,h), [g,h]_{{\mathfrak g}}), \; \; \; g, h \in {\mathfrak g},
$$
where $c$ 
is a skew-symmetric bilinear function on ${\mathfrak g}$, which takes its values in the space $=V$, and $[\cdot, \cdot]_{{\mathfrak g}}$ defines the Lie bracket of a Lie algebra ${\mathfrak g}$.
One can verify directly that the Jacobi identity for the bracket
$[\cdot, \cdot]_{\tilde {\mathfrak g}} $ is equivalent to the condition that the bilinear function is a cocycle, i.e. the following equality holds identically
$$
c([g,h]_{{\mathfrak g}},e)+c([h,e]_{{\mathfrak g}},g)+c([e,g]_{{\mathfrak g}},h)=0, \;\; \forall g,h,e \in {\mathfrak g},
$$
we assume that the initial bracket $[g,h]_{{\mathfrak g}}$ satisfies the Jacobi identity.  

Two extensions are called equivalent if there is an isomorphism of Lie algebras 
$f:  \tilde {\mathfrak g}_2 \to \tilde {\mathfrak g}_1$,
such that the following diagram is commutative
\begin{equation}
\label{equivalent}
\begin{CD}0 @>>> V@>{i_1}>>\tilde {\mathfrak g}_1 @>\pi_1>>{\mathfrak g}@>>>0\\
    @AAA @AA{Id}A @AA{f}A  @AA{Id}A @AAA\\
    0 @>>> V@>{i_2}>>\tilde {\mathfrak g}_2 @>\pi_2>>{\mathfrak g}@>>>0
\end{CD}
\end{equation}

A cocycle $c$ is called cohomologous to zero $c \sim 0$ if such a linear mapping exists  $\mu: {\mathfrak g} \to V$ such that $c(x,y)=\mu([x,y]_{\mathfrak g})$. In this situation, the cocycle $c$ is called a coboundary and is denoted by $c=d\mu$.  

Two cocycles are called cohomologous
$c \sim c '$ if their difference is cohomologous to zero $ c - c' \sim 0$. Cohomologous cocycles define equivalent central extensions. To prove this, it suffices to verify that the linear mapping
$$
f=Id + \mu: V \oplus \mathfrak{g} \to V \oplus \mathfrak{g}, \; f(v,g)=(v{+}\mu(g),g),
$$
is an isomorphism of Lie algebras in the diagram (\ref{equivalent}). The converse is also true \cite{Fu}. 

Note also that the cocycle $c '\sim 0$ cohomologous to zero defines an extension $ \tilde{\mathfrak g}' $ isomorphic to the direct sum of $ V \oplus {\mathfrak g} $ Lie algebras. Such a central extension is called trivial.

\begin{remark}
It may well happen that the Lie algebras $\tilde {\mathfrak g}_2$ and $\tilde{\mathfrak g}_1 $, corresponding to nonequivalent central extensions, are nevertheless isomorphic.The fact is that an isomorphism $ f $ from a commutative diagram (\ref{equivalent}) has to map $i_2(V) \subset \tilde {\mathfrak g}_2$ to $i_1(V) \subset \tilde {\mathfrak g}_1$ and induce identity mapping of
quotient algebras $Id: \tilde {\mathfrak g}_2/i_2(V) \to \tilde {\mathfrak g}_1/i_1(V)$. The absence of an isomorphism of $f$ with such additional properties does not mean the absence of isomorphism in general.  In the general case, an isomorphism is not required to translate $i_2(V)$ into $i_1(V)$

However, in the case of a nilpotent Lie algebra $ {\mathfrak g} $, the answer to the question of the isomorphism of its two different central extensions ${\mathfrak g}_2$ and $\tilde {\mathfrak g}_1$ is quite possible and constructive with some  $c_1$ and $c_2$,that we will show it the section \ref{s4}. 
\end{remark}

\section{Cohomology of positively graded and filtered Lie algebras}\label{s3}

Consider the standard cochain complex of an $n$-dimensional Lie algebra $ \mathfrak {g} $ with coefficients in a one-dimensional trivial module
 ${\mathbb K}$.
\begin{equation}
\label{cochain_complex}
\begin{CD}
\mathbb K @>{d_0{=}0}>> \mathfrak{g}^* @>{d_1}>> \Lambda^2 (\mathfrak{g}^*) @>{d_2}>>
\dots @>{d_{n-1}}>>\Lambda^{n} (\mathfrak{g}^*) @>>> 0
\end{CD}
\end{equation}
where the symbol $d_1: \mathfrak{g}^* \rightarrow \Lambda^2 (\mathfrak{g}^*)$
denotes the dual mapping to the Lie bracket
$[ \, , ]: \Lambda^2 \mathfrak{g} \to \mathfrak{g}$, 
with the differential $d$ (in fact, it is a set of mappings $d_p$)
is a derivation of exterior algebra 
 $\Lambda^*(\mathfrak{g}^*)$,
which continues $d_1$:
$$
d(\rho \wedge \eta)=d\rho \wedge \eta+(-1)^{deg\rho} \rho \wedge d\eta,
\; \forall \rho, \eta \in \Lambda^{*} (\mathfrak{g}^*).
$$
The relation $d^2=0$ is equivalent to the Jacobi identity in the Lie algebra $\mathfrak{g}$.

The cohomology of the complex $(\Lambda^*(\mathfrak{g}^*), d)$ is called
cohomology (with trivial coefficients) of a Lie algebra
$\mathfrak {g}$ and is denoted by  $H^*(\mathfrak{g}, {\mathbb K})$.

We do not define a cochain complex of the Lie algebra ${\mathfrak g}$ with values in arbitrary ${\mathfrak g}$-module $V$ referring the reader to \cite{Fu} for details, noting that if
 ${\mathfrak g}$-module $V$ is trivial ($gv=0, \forall g \in {\mathfrak g}, \forall v \in V$) and $\dim V=m$ then there are isomorphisms of ${\mathfrak g}$-modules and cohomology
$$
V \cong \underbrace{{\mathbb K}\oplus \dots \oplus {\mathbb K}}_{m},\;
H^q({\mathfrak g}, V)\cong \underbrace{H^q({\mathfrak g},{\mathbb K})\oplus \dots \oplus H^q({\mathfrak g},{\mathbb K})}_{m}=(H^q({\mathfrak g}, {\mathbb K}))^m.
$$

For any ${\mathbb N}$-graded Lie algebra  
$\mathfrak{g}=\oplus_{i=1}^{+\infty}\mathfrak{g}_i$ its exterior algebra $\Lambda^* \mathfrak{g}$
can be endowed with the second grading 
$\Lambda^* \mathfrak{g} =
\bigoplus_{i=1}^{+\infty} \Lambda^*_{i} \mathfrak{g}$, where
$\Lambda^p_{i} \mathfrak{g}$ is a linear span of monomials
$\xi_1 \wedge \xi_2 \wedge \dots \wedge\xi_p$ such that 
$$
\xi_1 \in \mathfrak{g}_{\alpha_1}, \xi_2 \in \mathfrak{g}_{\alpha_2},
\dots,  \xi_p \in \mathfrak{g}_{\alpha_p}, \;
\alpha_1{+}\alpha_2{+}\dots{+}\alpha_p=i.
$$
The space of skew-symmetric $p$-functions
$\Lambda^p (\mathfrak{g}^*)$  is also endowed with a second grading 
$\Lambda^p (\mathfrak{g}^*) =
\bigoplus_{\lambda} \Lambda^p_{(\lambda)} (\mathfrak{g}^*)$. 
The subspace
$\Lambda^p_{(\lambda)} (\mathfrak{g}^*), \lambda \in {\mathbb N},$ is defined as
$$ 
\Lambda^p_{(\lambda)} (\mathfrak{g}^*)= \left\{ \omega
\in \Lambda^p (\mathfrak{g}^*)
 \; | \; \omega(v)=0, \: \forall v \in
 \Lambda^p_{(\mu)} (\mathfrak{g}), \: \mu \ne \lambda \right\}.
$$
We will consider the cohomology of only finite-dimensional  ${\mathbb N}$-graded  Lie algebras, and therefore
the sign of the direct sum in the preceding formulas denotes the usual finite direct sum of subspaces. 

The second grading is compatible with the differential
$d$ and with the exterior product
$$
d \Lambda^p_{(\lambda)} (\mathfrak{g}^*)
\subset \Lambda^{p{+}1}_{(\lambda)} (\mathfrak{g}^*), \qquad
\Lambda^{p}_{(\lambda)} (\mathfrak{g}^*) \cdot
\Lambda^{q}_{(\mu)} (\mathfrak{g}^*) \subset
\Lambda^{p{+}q}_{(\lambda{+}\mu)} (\mathfrak{g}^*)
$$
Homogeneous forms from the subspace $\Lambda^p_{(\lambda)} (\mathfrak{g}^*)$ we will call $p$-forms of weight  $\lambda$ in the sequel, respectively closed  $p$-forms of weight $\lambda$ will be called $p$-cocycles of weight $\lambda$, a similar rule would be for coboundaries.

External product in $\Lambda^*(\mathfrak {g}^*)$
induces the structure of bigraded algebra in cohomology
 $H^*(\mathfrak{g})$
$$
H^{p}_{(\lambda)} (\mathfrak{g}) \otimes
H^{q}_{(\mu)} (\mathfrak{g}) \to
H^{p{+}q}_{(\lambda{+}\mu)} (\mathfrak{g}).
$$

\begin{example}
\label{cohomol_m_0_n}
The cochain complex $(\Lambda^*({\mathfrak m}_0(n)), d)$ is generated by $a^1,b^1,a^2,\dots, a^n$ with the differential defined by
$$
da^1=db^1=0, da^2=a^1\wedge b^1, da^{i}=a^1\wedge a^{i-1}, 3 \le i \le n.
$$
Its cohomology (i.e. the cohomology $H^*({\mathfrak m}_0(n),{\mathbb K})$ of the Lie algebra ${\mathfrak m}_0(n)$) is long known. In particular, dimension of the space  
$H^2({\mathfrak m}_0(n),{\mathbb K})$ is $k=\left[ \frac{n}{2}\right]$ and is the linear span of the following system of homogeneous $2$-cycles of odd weights $3,5,\dots,2k+1$:
\begin{equation}
\label{H^2_cohomol_m_0}
c_{2q+1}=-b^1\wedge a^{2q}+\sum_{i=2}^{q}(-1)^ia^i\wedge a^{2q+1-i}, \; q=1,\dots,k.
\end{equation}
\end{example}

We now consider a positively filtered Lie algebra ${\mathfrak g}$ with a filtration of length $m$
$$
F^1{\mathfrak g}={\mathfrak g}\supset F^2{\mathfrak g} \supset \dots \supset 
F^m{\mathfrak g} \supset \{0\}
$$
We define by means of it the increasing filtration  $\tilde F$ of the  cochain complex $(\Lambda^*({\mathfrak g}^*), d)$ of the Lie algebra
${\mathfrak g}$ with coefficients in a trivial one-dimensional module ${\mathbb K}$. For the natural numbers $p, k$, we introduce a linear subspace 
\begin{equation}
\label{L*g*filtr}
\tilde F^k\Lambda^p({\mathfrak g})=\left\{\omega \in \Lambda^p({\mathfrak g}) | \omega(\xi_1,\dots,\xi_p)=0, \xi_i {\in} F^{k_i}{\mathfrak g},
k_1+\dots +k_p>k \right\}
\end{equation}
in the space of all skew-symmetric $p$-forms $\Lambda^p({\mathfrak g})$.

It is easy to check the invariance of the filtering $\tilde F$ with respect to the differential $d$ and the nesting relation of the subspaces
$$
 \tilde F^0 \Lambda^p (\mathfrak{g}^*)=0,  \tilde F^k \Lambda^p (\mathfrak{g}^*) \subset  \tilde F^{k+1} \Lambda^p (\mathfrak{g}^*),\;
d \tilde F^k \Lambda^p (\mathfrak{g}^*)
\subset\tilde F^k \Lambda^{p{+}1}(\mathfrak{g}^*).
$$
\begin{definition}
We will say that the $p$ -form $\omega \in \Lambda^p({\mathfrak g})$ has filtration $k$ and write $\phi(\omega)=k$, if 
$$
\omega \in F^k \Lambda^p (\mathfrak{g}^*), \; \omega \notin 
F^{k-1} \Lambda^{p}(\mathfrak{g}^*).
$$
\end{definition}
In the subsequent sections, we also need one increasing filtration of the dual space ${\mathfrak g}^*$ to the nilpotent Lie algebra ${\mathfrak g} $. This filtration is constructed recurrently, but, as we will see, is connected with the filtration of the original Lie algebra ${\mathfrak g} $ with ideals of the lower central series.

Define a chain of embedded in each other subspaces in ${\mathfrak g}^*$
$$
L_0=0 \subset L_1 \subset L_2 \subset \dots \subset L_i \subset \dots 
$$
where the subspace $L_i$ is defined by the following
\begin{equation}
\label{g*filtr}
L_i=\{ \rho \in{\mathfrak g}^*: d \rho \in \Lambda^2 (L_{i-1})\},\; i\ge1.
\end{equation} 

Obviously, the subspace
$L_1= \ker d$ coincides with the subspace of closed $1$-forms of a cochain complex (\ref {cochain_complex}) of a Lie algebra  ${\mathfrak g}$.

\begin{propos}
\label{propos_L_i}
The subspace $L_i \subset {\mathfrak g}^*$ is the annihilator of the ideal ${\mathfrak g}^{i+1}$ for $0 \le i \le s$. 
\end{propos}
\begin{proof}
We will prove the statement by induction on $i$. The basis of induction is obvious: $L_0=0$ is the annihilator of the whole algebra ${\mathfrak g}={\mathfrak g}^1$. Suppose, according to the inductive hypothesis, the subspace $L_{i-1}$ is the annihilator of the ideal ${\mathfrak g}^{i}$. 
Then $d\rho \in \Lambda^2(L_{i-1})$ if and only if $d\rho$ is annihilated by the subspace
${\mathfrak g}\wedge{\mathfrak g}^{i}$ or in other words $d\rho( \xi,\eta)=-\rho([\xi,\eta])=0$ for all  $\xi \in{\mathfrak g}$ and
$\eta \in{\mathfrak g}^{i}$. Thus $d\rho \in \Lambda^2(L_{i-1})$ if and only if  the linear function $\rho \in {\mathfrak g}^*$ belongs to the  annihilator of the  ${\mathfrak g}^{i+1}$.
\end{proof}
From the proved proposition it follows that the chain length $\{L_i\} $ is finite. Its length is equal to the nilpotency index $s$: $L_s={\mathfrak g}^*$. 

\section{A recurrent method for constructing nilpotent Lie algebras}\label{s4}

Consider a (non-Abelian) nilpotent Lie algebra
$\tilde{\mathfrak g}$. It has a non-trivial cente $Z(\tilde {\mathfrak g})$ 
(the last nontrivial ideal ${\mathfrak g}^s$ of its lower central series) belongs to the center $Z(\tilde {\mathfrak g})$).
Consider the quotient Lie algebra ${\mathfrak g}=\tilde{\mathfrak{g}}/ Z(\tilde {\mathfrak g})$ and its corresponding central extension
$$
0 \to Z(\tilde {\mathfrak g}) \to  \tilde{\mathfrak g}  \to {\mathfrak g} \to 0. 
$$
This central extension is defined by some cocycle $\tilde c$  from $H^2({\mathfrak g}, Z(\tilde {\mathfrak g}))$. 

Recall also that, according to the Dixmier theorem \cite{Dixm}, all the cohomology groups $H^i({\mathfrak g},V)$ of an arbitrary finite-dimensional nilpotent Lie algebra ${\mathfrak g}$ with coefficients in any trivial ${\mathfrak g}$-module $V$ are non-trivial.

Fix a basis $e_1,\dots,e_m$ of the ideal $Z(\tilde {\mathfrak g})$ where   $\dim Z(\tilde {\mathfrak g})=m$.This will give the opportunity to write the cocycle $\tilde c $ in the corresponding coordinates
$
\tilde c=(\tilde c_1,\dots,\tilde c_m).
$
The component $\tilde c_l$  of the cocycle $\tilde c$ has a simple meaning, it is the differential of a linear functional  $e^l$ from the basis that is dual to the basis $e_1,\dots,e_m$ of the ideal $Z(\tilde {\mathfrak g})$ in the dual space $Z(\tilde {\mathfrak g})^*$
$$
de^l=\tilde c_l, \; l=1,\dots,m.
$$
On the other hand, fixing a basis of the ideal $Z(\tilde {\mathfrak g})$  is nothing but the presentation of explicit isomorphism of trivial ${\mathfrak g}$-modules
$Z(\tilde {\mathfrak g}) \to  \underbrace{{\mathbb K}\oplus \dots \oplus{\mathbb K}}_{m}$.

Further, taking into account the convenience of specific calculations, we will define a cocycle $\tilde c$ from $H^2({\mathfrak g}, Z(\tilde {\mathfrak g}))$ by means of its image under the isomorphism $H^2({\mathfrak g}, Z(\tilde {\mathfrak g})) \to H^2({\mathfrak g}, {\mathbb K}^m)$ i.e. by the set $(\tilde c_1,\dots, \tilde c_m)$ of cocycles from $H^2({\mathfrak g},{\mathbb K})$.

Consider two central extensions $\tilde {\mathfrak g}$ and $\tilde {\mathfrak g}'$ of the same nilpotent Lie algebra${\mathfrak g}$ using the same vector space $V$ such that $Z(\tilde {\mathfrak g}) \cong Z(\tilde {\mathfrak g}') \cong V$. 
Let there also be an isomorphism $f: \tilde {\mathfrak g}' \to \tilde {\mathfrak g}$. 
The equality holds
$
f\left(Z( \tilde {\mathfrak g}')\right)=Z\left(\tilde {\mathfrak g}\right)
$
and we have a commutative diagramm
\begin{equation}
\label{isom}
\begin{CD}0 @>>> Z\left(\tilde {\mathfrak g}\right)@>{i_1}>>\tilde {\mathfrak g} @>\pi_1>>{\mathfrak g}@>>>0\\
    @AAA @AA{\Psi}A @AA{f}A  @AA{\Phi}A @AAA\\
    0 @>>> Z( \tilde {\mathfrak g}')@>{i_2}>>\tilde {\mathfrak g}' @>\pi_2>>{\mathfrak g}@>>>0,
\end{CD}
\end{equation}
where by $\Psi$ we denote the isomorphism of vector spaces $Z\left(\tilde {\mathfrak g}'\right)$ and $Z\left(\tilde {\mathfrak g}\right)$. The symbol $\Phi$ denotes the automorphism of the Lie algebra ${\mathfrak g}$, induced by the isomorphism $f$.
\begin{example}
\label{collision}
Let ${\mathfrak g}={\mathfrak m}_0(2)\oplus {\mathbb K}$ be the direct sum of Lie algebras.  Consider as 
$\tilde {\mathfrak g}=\tilde {\mathfrak g}'={\mathfrak g}\oplus{\mathbb K}$  its one-dimensional trivial central extension. From the formal point of view, the subspace $i_1({\mathbb K}) = i_2({\mathbb K})$ does not have to be invariant with respect to an arbitrary automorphism  $f: \tilde {\mathfrak g}' \to  \tilde {\mathfrak g}$. 
\end{example}
Despite the deliberate artificiality of the above example, we must state that when the images $i_1(V)$ and $i_2(V)$ do not coincide with the centers
$Z\left(\tilde {\mathfrak g}'\right)$, $Z\left(\tilde {\mathfrak g}\right)$
of constructed Lie algebras $\tilde {\mathfrak g}'$ and $\tilde {\mathfrak g}$, we cannot guarantee the compatibility of an arbitrary isomorphism of $f$ with embeddings $i_1$ and $i_2$. However, the following lemma is true
\begin{lemma}
\label{main_propos}
Let $\{ \tilde c_1,{\dots}, \tilde c_m \}$ and $\{ \tilde c_1',{\dots}, \tilde c_m'\}$ be two sets of cocycles from $Z^2({\mathfrak g}, {\mathbb K})$. 
There is an isomorphism $f:\tilde {\mathfrak g}'  \to \tilde {\mathfrak g}$ of the corresponding central extensions such that  $f(i_2({\mathbb K}^m))=i_1({\mathbb K}^m)$ if and only if there are: 
a) non-degenerate number matrix  $A_{\Psi}$ and b) an automorphism $\Phi$ of the algebra
${\mathfrak g}$ such that equality holds
\begin{equation}
\label{ext_isom}
(\tilde c_1,\tilde c_2,\dots,\tilde c_m)A_{\Psi}=
(\Phi^*(\tilde c_1'),\Phi^*(\tilde c_2'),\dots,\Phi^*(\tilde c_m')).
\end{equation}
In the above formula, we denoted by the symbol $\Phi^*$ the action of the automorphism $\Phi$ on the space of two-dimensional cocycles $Z^2({\mathfrak g}, {\mathbb K})$.   
\end{lemma}
\begin{proof}
Choose bases ${e^1}',\dots,{e^m}'$ and $e^1,\dots,e^m$ in the dual spaces
$Z(\tilde {\mathfrak g}')^*$ and $Z(\tilde {\mathfrak g})^*$ respectively.
A Lie algebras isomorphism $f: \tilde {\mathfrak g}' \to \tilde {\mathfrak g}$ induces the isomorphism of $d$-algebras $f^*: ({\tilde {\mathfrak g}}^*, d) \to (({\tilde {\mathfrak g}}')^*, d)$ and in particular the linear spaces isomorphism $\Psi: Z(\tilde {\mathfrak g}) \to Z(\tilde {\mathfrak g}')$. Write in line
$(\Psi({e^1}'),\dots,\Psi({e^m}'))$  images of basis vectors from $Z({\tilde {\mathfrak g}}')^*$  in  $Z(\tilde {\mathfrak g})^*$. Introduce the transition matrix $A_{\Psi}$ from the basis
$e^1,e^2,\dots,e^m$ to the basis $\Psi({e^1}'),\Psi({e^2}'),\dots,\Psi({e^m}')$
\begin{equation}
(e^1,e^2,\dots,e^m)A_{\Psi}=
(\Psi({e^1}'),\Psi({e^2}'),\dots,\Psi({e^m}')).
\end{equation}
Applying the differential $ d $ to both sides of this equality and tacking into account
$$
d{e^i}=\tilde c_i, \; d\Psi({e^i}')=df^*({e^i}')=f^*d{e^i}'=\Phi^*d{e^i}'=\Phi^*(\tilde c_i'), i=1,\dots, m
$$ 
we get the required equality (\ref{ext_isom}).

In the other direction, the proposition is proved by repeating the above reasoning.
\end{proof}
\begin{propos}
\label{ext_cohom}
Let $\{ \tilde c_1,{\dots}, \tilde c_m \}$ and $\{ \tilde c_1',{\dots}, \tilde c_m'\}$ be two sets of cocycles from $Z^2({\mathfrak g}, {\mathbb K})$ such that
$$
\{ \tilde c_1',{\dots}, \tilde c_m'\}=\{ \tilde c_1+d\mu_1,{\dots}, \tilde c_m+d\mu_m\}.
$$ 
Then the corresponding central extensions are isomorphic  $\tilde {\mathfrak g}'\cong \tilde {\mathfrak g}$.
\end{propos}

Let us formulate a very natural question about the properties of the Lie algebra $ \tilde{\mathfrak g}$, obtained as the central extension of some nilpotent Lie algebra ${\mathfrak g}$: will it be nilpotent, and if so, what will it have nilpotency index $s(\tilde {\mathfrak g})$? It is clear that the answer to this question must be given in terms of the set of cocycles $ (c_1, \dots, c_m)$ which defined the central extension $\tilde {\mathfrak g}$.

The answer to a similar question about the properties of the Lie algebra 
$\tilde {\mathfrak g} $, obtained using the deformation $\Psi $ of the nilpotent Lie algebra ${\mathfrak g} $, can be found in the classical work Vergne \cite{V} in terms of the cohomology of $H^2({\mathfrak g}, {\mathfrak g}) $ of the Lie algebra ${\mathfrak g}$ with coefficients in the adjoint representation. In a sense, the problem solved Vergne is more general, and the answer to our question can in principle be obtained in the form of its corollary. But we will formulate the answer directly, considering its importance for specific applications.

Following \cite{V}, we consider the filtration of a finite-dimensional nilpotent Lie algebra  $\tilde{\mathfrak g}$ by the ideals $\tilde {\mathfrak g}^k$ of its lower central series.  
The last nontrivial ideal $\tilde{\mathfrak g}^s $ belongs to the center
$Z(\tilde {\mathfrak g})$ of the Lie algebra  $\tilde {\mathfrak g}$.

Consider the quotient Lie algebra ${\mathfrak g}=\tilde{\mathfrak{g}}/\tilde {\mathfrak g}^s$ and the corresponding central extension
$$
0 \to \tilde {\mathfrak g}^s \to  \tilde{\mathfrak g}  \to {\mathfrak g} \to 0. 
$$
This central extension is defined by some cocycle $\tilde c$ from $H^2({\mathfrak g}, \tilde {\mathfrak g}^s)$.
Fix a basis $e_1^s, \dots, e_m^s$ of the ideal ${\mathfrak g}^s$ and let us write the cocycle $\tilde c$ in the corresponding coordinates
$
\tilde c=(\tilde c_1,\dots,\tilde c_m).
$
The first question: what can be said about the set of cocycles $\tilde c_1,\dots,\tilde c_m$?

The second question is how to choose a set of $\tilde c_1, \dots, \tilde c_m$ cocycles from $H^2({\mathfrak g}, {\mathbb K})$ so that the corresponding $m$-dimensional central extension of $\tilde {\mathfrak g}$ its nilpotency index $s=s (\tilde {\mathfrak g}) $ would be one greater than the original one of the  Lie algebra ${\mathfrak g}$: $s=s(\tilde {\mathfrak g})=s({\mathfrak g})+1$ and the dimension of the $s$th ideal ${\mathfrak g}^s$ of the lower central series would be exactly $m$?

\begin{definition} Let ${\mathfrak g}$ be a nilpotent Lie algebra. 
We say that a set of cocycles $\{\tilde c_1,{\dots}, \tilde c_m \}$ from $H^2({\mathfrak g},{\mathbb K})$  has filtration $s$  if cocycles of this set are linearly independent modulo subspace $F^{s-1}H^2({\mathfrak g}^*,{\mathbb K})$.
\end{definition}

Recall that the subspace $F^{s-1}H^2({\mathfrak g}^*, {\mathbb K}) $ consists of cohomology classes of two-dimensional cocycles that vanish at
all subsets ${\mathfrak g}^{k_1} \times {\mathfrak g}^{k_2}$ for $k_1+k_2 > s-1$.
\begin{theorem}
\label{main1}
Let ${\mathfrak g}$  be a nilpotent Lie algebra with nil-index  $s-1$ and let  $\tilde c_1,\dots,\tilde c_m$ be a set of cocycles from
$H^2({\mathfrak g}, {\mathbb K})$. The Lie algebra $\tilde{\mathfrak g} $, defined by the corresponding central extension corresponding to this set, is a nilpotent Lie algebra of nil-index $s$ and $\dim{ \tilde {\mathfrak g}^s}=m$ if and only if the set of cocycles   $\tilde c_1,\dots,\tilde c_m$ has filtration $s$.
\end{theorem}
\begin{proof}
Let us prove the necessity of the condition formulated in the theorem, for which we choose in the ideal $\tilde {\mathfrak g}^s$ of the Lie algebra $\tilde {\mathfrak g}$ obtained as a central extension ${\mathfrak g}$ with a basis $e_1^s,\dots, e_m^s$. The original Lie algebra ${\mathfrak g} $ had the nilpotency index $s-1$, so all these vectors were added with a central extension and, thus, we can assume that if necessary we can replace the base and hence  
$de^1_s=\tilde c_1,\dots, de^m_s=\tilde c_m$. We denote by symbols $e^1_s,\dots, e^m_s$ the linear functions from the basis which is dual to $e_1^s,\dots, e_m^s$.
Suppose there is a linear combination
$ \alpha_1 \tilde c_1+\dots + \alpha_m \tilde c_m$, which vanishes at
all subsets
${\mathfrak g}^{k_1} \times {\mathfrak g}^{k_2}$ for $k_1+k_2 > s-1$. This in particular means that 
$$
(\alpha_1\tilde c_1+\dots+\alpha_m\tilde c_m)({\mathfrak g}^1, {\mathfrak g}^{s-1})=(\alpha_1e^1_s+\dots+\alpha_me^m_s)([{\mathfrak g}^1, {\mathfrak g}^{s-1}])=0.
$$
Thus, this linear combination identically vanishes on the whole ideal ${\mathfrak g}^s$ that contradicts the choice of functionals 
$e^1_s,\dots, e^m_s$.

Now we prove the statement of the theorem in a different direction. We use filtering (\ref{g*filtr}). Obviously, there is an inclusion
$$
F^{s-1}\Lambda^2({\mathfrak g})\subset \Lambda^2(L_{s-1}),
$$
where $L_{s-1}$ denotes the annihilator of the ideal  $\tilde {\mathfrak g}^s$ in the dual space $\tilde {\mathfrak g}^*$. Thus
$1$-forms $e^1,\dots,e^m$ from $\tilde {\mathfrak g}^*$ such that $de^1=\tilde c_1,\dots, de^m=\tilde c_m$ belong to the subspace $L_s=\tilde {\mathfrak g}^*$ (the annihilator of the ideal) $\tilde {\mathfrak g}^{s+1}=\{0\}$).

The subspace of $F^{s-1}\Lambda^2({\mathfrak g})$, generally speaking, does not coincide with $\Lambda^2(L_{s-1})$ and a linear independence of $2$-form modulo $F^{s-1}\Lambda^2({\mathfrak g})$  does not imply their independence modulo the larger subspace $\Lambda^2(L_{s-1})$, but in a situation where these $2$-forms are cocycles, this becomes a valid statement. To prove it, we construct the basis of the subspace  $L_{s-1}$. As the first step of its construction, we choose the basis $e_1^1,\dots, e_1^{m_1}$ of the subspace $L_1$ of closed $1$-forms, let us add forms
 $e_2^1,\dots, e_2^{m_2}$ for a basis $L_2$, and so on. Last in this basis we add linear functions $e_{s-1}^1,\dots, e_{s-1}^{m_{s-1}}$.  

It is easy to see that such a basis is a dual basis to some basis of an extendable Lie algebra ${\mathfrak g}$ 
$$
e_1^1,\dots,e_{j_1}^1,e_1^2,\dots,e^2_{j_2},\dots,e_1^{s-1},\dots,e_{j_{s-1}}^{s-1} ,
$$
where the vectors  $e_1^{s-1},\dots,e_{j_{s-1}}^{s-1}$ form the basis ${\mathfrak g}^{s-1}$, and vectors $e_1^{s-2},\dots,e_{j_{s-2}}^{s-2}$ complement $e_1^{s-1},\dots,e_{j_{s-1}}^{s-1}$ to a basis of the ideal ${\mathfrak g}^{s-2}$ and so on. The latter are added vectors  $e_1^1,\dots,e_{j_1}^1$ that complement the already constructed basis of the commutant  $[{\mathfrak g},{\mathfrak g}]$ to the basis of the entire Lie algebra ${\mathfrak g}$.

Each cocycle $\tilde c_n \in \Lambda^2(L_{s-1}), n=1, \dots, m$ of our set can be written as follows
$$
\tilde c_n = \sum_{j=1}^{m_1}\sum_{l=1}^{m_{s-1}} \alpha_{j,l}^ne_1^j \wedge e_{s-1}^l + \tilde c_{2,n}, \tilde c_{2,n}\in \Lambda^2 (L_{s-2}),\;  n=1,\dots,m,
$$
Indeed, choosing in a different way linear functions
$e_{s-1}^1,\dots, e_{s-1}^{m_{s-1}}$ we may assume that the cocycle $\tilde c_n$ as an arbitrary element of filtration $s-1$  can be written in a following way for some $q $ such that $2\le 2q \le j_{s-1},$ and $q$ depends on $n$:
$$
\tilde c_n=\gamma_r\sum_{k=1}^{q}e_{s-1}^{2k-1}\wedge e_{s-1}^{2k}+ \sum_{ i=1}^{j_{s-1}} \omega_i^n\wedge e_{s-1}^i+\tilde c_2,  \omega_i^n \in L_{s-2}, 
1 \le i \le j_{s-1},
\tilde c_{2,n} \in \Lambda^2(L_{s-2}).
$$
Express the differential $d\tilde c_n$
$$
d \tilde c_n=\gamma_n \sum_{k=1}^{q}(de_{s-1}^{2k-1}\wedge e_{s-1}^{2k}-e_{s-1}^{2k-1}\wedge de_{s-1}^{2k})+ \sum_{i=1}^{j_{s-1}} d\omega_i^n\wedge e_{s-1}^i+\Omega_n, \Omega_n \in \Lambda^3(L_{s-2}).
$$
If $\gamma_n \ne 0$, then $d(e^{2q}_{s-1} + \omega_{2q}^n) = 0$. Where it follows that $e^{2q}_{s-1}+\omega_{2q}^n \in L_1$ or $e^{2q}_{s-1} \in L_{s-2}$, which is contrary to the choice of the basis $e^1_{s-1},\dots,e^{j_{s-1}}_{s-1}$.

In the case $\gamma_n=0$ all functionals $\omega_i^n$ have to be closed and hence $\omega_i^n=\sum_{j=1}^{j_1}\alpha_{i,j}^ne^j_1$ for $1 \le i \le j_{s-1}$ which gives the required presentation for $\tilde c_n$.

By definition of the central expansion and the differential $d$ we have
$$
[e_{r}^1,e_i^{s-1}]=\sum_{j=1}^m\alpha_{r,i}^je_j^s, \; 1 \le r \le j_1, 1\le i \le j_{s-1},
$$ 
where $e_1^s,\dots,e_m^s$ are vectors added while central extension.
Consider the linear span of a system of $N=j_1j_{s-1}$ vectors
$$
[e_{r}^1,e_i^{s-1}], 1 \le r \le j_1, 1\le  i \le j_{s-1}.
$$ 
Let us prove that its dimension is
$m$.  Indeed, we enumerate the set of pairs of natural numbers in some (standard) way.
with numbers from $1$ to $N=j_1j_{s-1}$. The corresponding numbering index is denoted by $q=q(r,i)$.
Write the matrix $A=(\alpha_{q}^j)=(\alpha_{(r,i)(q)}^j)$. It coincides with the matrix, where the coordinates of the vectors  
$[e_{r}^1,e_i^{s-1}]$ with respect to the basis $e_1^s,\dots,e_m^s$ are in columns. However its $j$th 
row  consists of numbers of the form $\alpha_{r, i}^j$, written in one line using our ordering with the corresponding index
 $q, 1 \le q \le N=j_1j_{s-1}$. The rank of the row system of such a matrix is $m$ by the condition, which means that the rank of its column system is also $m$. We thereby proved that the dimension of the ideal $\tilde{\mathfrak g}^s$ of the central extension $\tilde {\mathfrak g}$ is $m$.
\end{proof}

Consider two central extensions $\tilde {\mathfrak g} $ and $ \tilde{\mathfrak g} '$ of the same Lie algebra ${\mathfrak g}$ of nil-index $s-1$. Assume that  $\tilde {\mathfrak g}$ and $\tilde {\mathfrak g}'$ both are nilpotent Lie algebras with the same nil-index $s$. 

Let there also be an isomorphism $f: \tilde {\mathfrak g}' \to \tilde {\mathfrak g}$. The equality   $f\left(\tilde {\mathfrak g}'^k\right)=\tilde {\mathfrak g}^k$  holds for all ideals of the lower central series. Where it follows in particular that
$$
f\left(( \tilde {\mathfrak g}')^s\right)=\tilde {\mathfrak g}^s,
$$
we have a commutative diagram
\begin{equation}
\label{isom}
\begin{CD}0 @>>> \tilde {\mathfrak g}^s@>{i_1}>>\tilde {\mathfrak g} @>\pi_1>>{\mathfrak g}@>>>0\\
    @AAA @AA{\Psi}A @AA{f}A  @AA{\Phi}A @AAA\\
    0 @>>> (\tilde {\mathfrak g}')^s@>{i_2}>>\tilde {\mathfrak g}' @>\pi_2>>{\mathfrak g}@>>>0,
\end{CD}
\end{equation}
where the symbol $\Psi$ denotes an isomorphism of vector spaces $\tilde {\mathfrak g}^s$ and $(\tilde {\mathfrak g}')^s$.  $\Phi$ stands for some automorhism of the Lie algebra ${\mathfrak g}$.
\begin{remark}
The restrictions that we impose on the choice of the set of cocycles $ \tilde c_1, \dots, \tilde c_m $ exclude collisions, shown by the example of \ref{collision}: now all isomorphisms are compatible with the embeddings $i_1$ and $i_2$.
\end{remark}
\begin{theorem}
\label{main2}
Let ${\mathfrak g}$ be a nilpotent Lie algebra of nil-index $s-1$ and also $\{ \tilde c_1,{\dots}, \tilde c_m \}$ and $\{ \tilde c_1',{\dots}, \tilde c_m'\}$ be two sets of cocycles of filtration $s$  in $H^2({\mathfrak g}, {\mathbb K})$. 
They define isomorphic central extensions  $\tilde {\mathfrak g}$ and $\tilde {\mathfrak g}'$  if and only if the linear spans  $\langle \tilde c_1,{\dots}, \tilde c_m \rangle$ and $\langle \tilde c_1',{\dots}, \tilde c_m'\rangle$ lie in the same orbit of the linear action of the automorphism group ${\rm Aut}({\mathfrak g})$  on the space   $H^2({\mathfrak g}, {\mathbb K})$.   
\end{theorem}
\begin{proof}
Let us prove this proposition in one direction; the converse is left as an elementary exercise for the reader. Let $L$ and $L '$ be two linear subspaces
in cohomology $H^2({\mathfrak g})$ such that  $\Phi(L')=L$ where $\Phi \in Aut({\mathfrak g})$ denotes some automorphism of the Lie algebra  ${\mathfrak g}$. We assume that both subspaces are given as linear spans of two basic sets of cocycles $\langle \tilde c_1,{\dots}, \tilde c_{j_k} \rangle$ and $\langle \tilde c_1',{\dots}, \tilde c_{j_k}'\rangle$ of filtration $s$. We apply the standard formulas for $\Phi$-action on bilinear forms.
$$
(\Phi \cdot \tilde c_l')(x,y)= \tilde c_l'(\Phi^{-1} x, \Phi^{-1} y), \forall x,y \in {\mathfrak g}, \; l=1,\dots, j_k, \; \Phi \in Aut({\mathfrak g}).
$$
Bilinear functions  $\Phi \cdot \tilde c_1', \dots, \Phi \cdot \tilde c_{j_k}'$ form a basis in $L$ as well as $ \tilde c_1,{\dots}, \tilde c_{j_k} $. Thus, there is an automorphism $\psi: L \to L$ such that
$$
\psi \left( \Phi \cdot \tilde c_1'\right)=\tilde c_1, \dots, \psi \left( \Phi \cdot \tilde c_{j_k}'\right)=\tilde c_{j_k}.
$$ 
Define $\tilde {\mathfrak g}=L\oplus{\mathfrak g}$ and $\tilde {\mathfrak g}'=L'\oplus{\mathfrak g}$ as vector spaces. Set linear mapping $f: \tilde {\mathfrak g}' \to \tilde {\mathfrak g}$ by the formula $f((a,g))=(\Psi(a),\Phi(g))$, where $\Psi(a)=\psi \cdot (\Phi \cdot a)$ and $\Phi$ is an automorphism of the Lie algebra ${\mathfrak g}$. Evidently that this mapping is an automorphism of vector spaces.

Check the compatibility of the mapping $f$ with the Lie brackets of Lie algebras $\tilde {\mathfrak g}'$ and  $\tilde {\mathfrak g}$
$$
\begin{array}{c}
f([(a,g),(b,h)]_{\tilde {\mathfrak g}'})=f((\tilde c'(g,h),[g,h]_{\mathfrak g}))
=(\Psi\tilde c'(g,h),\Phi([g,h]_{\mathfrak g}))=\\
=(\tilde c(\Phi g,\Phi h),[\Phi g, \Phi h ]_{\mathfrak g})=[(\Psi a,\Phi g),(\Psi b,\Phi h)]_{\tilde {\mathfrak g}}=[f(a,g),f(b,h)]_{\tilde {\mathfrak g}}
\end{array}
$$
\end{proof}

Fix the main result of this section. The recursive method of successive central extensions is constructed for defining and classifying finite-dimensional nilpotent Lie algebras.
We note its monotonicity: 1) at each step and the dimension and nilpotency index of an expandable Lie algebra increase; 2) two non-isomorphic nilpotent Lie algebras can no longer have isomorphic central extensions with our restrictions on sets of cocycles. We can illustrate the last remark in a way from genealogy: the set of "descendants" (i.e., the set of consecutive central extensions) of two different algebras from our list cannot intersect \cite{Mill5}, and there are algebras that "have no progeny," example of such an algebra ${\mathfrak m}_1(2k-1)$ we give in the section \ref{s6}.

How effective is this method? Is it possible to classify with its help nilpotent Lie algebras of small dimensions? The Morozov \cite{Mor} classification of $6$-dimensional nilpotent Lie algebras has long been known and there are several classification lists of $7$-dimensional complex Lie algebras. However, the hope of success of such a classification in subsequent dimensions will be very restrained: it suffices to recall the classification of metabelian Lie algebras \cite{Gaug, GalTim}, which we will have to include as an integral part of this classification.
The maximum dimension in which metabelian Lie algebras are classified is $9$ to date \cite{GalTim}.
Despite all these difficulties,
it would still be useful to implement the constructed method for the classification of $7$-dimensional and $8$-dimensional nilpotent Lie algebras. We will postpone such research to subsequent publications.

\section{Rigid central extensions of nilpotent Lie algebras}\label{s5}
\begin{definition}
Set of cohomology classes $\tilde c= (\tilde c_1,\tilde c_2,\dots,\tilde c_m)$ from the space  $(H^2({\mathfrak g}, {\mathbb K}))^m$ of nilpotent Lie algebra  ${\mathfrak g}$  is called {\it geometrically rigid} if such a neighborhood exists
$U(\tilde c) \subset (H^2({\mathfrak g}, {\mathbb K}))^m$ (in the standard topology of a finite-dimensional space) that for any other set of cocycles $\tilde c '$ from this neighborhood the corresponding Lie algebra
 ${\mathfrak g}_{\tilde c'}$ constructed as a central extension  ${\mathfrak g}$ over the set  ${\tilde c'}$ over the set $ {\tilde c'} $ will be isomorphic to a Lie algebra ${\mathfrak g}_{\tilde c}$.
\end{definition}

We will immediately clarify that in algebraic literature more often, when it comes to the orbits of an algebraic group, the Zarissky topology is considered and usually the openness of the orbit is understood precisely in the sense of this topology.
We will now use an equivalent geometric approach and, accordingly, consider the standard Euclidean topology of a finite-dimensional space to visually describe the orbit spaces of the actions we need for the algebraic subgroups of $GL_2 $ on some cohomology spaces $H^2({\mathfrak g}, {\mathbb R})^m$ of small dimensions -- a similar geometric approach was considered in \cite{Gorb}. It is the real classification that is our main goal, in the light of its various geometric applications. The study of the orbit space of the action of an algebraic group on an affine variety is the subject of the classical theory of invariants, but the goal of this article is more modest: {\it we want to depict orbits that are interesting to us using images and means of elementary low-dimensional geometry.}

Since the natural action of the group ${\rm Aut}({\mathfrak g})$ on the two-dimensional cohomology space $H^2({\mathfrak g}, {\mathbb K}) $ is algebraic, the following statement is true.
\begin{propos}
\label{finite_rigid}
Let the orbit space of the action $GL_m \times {\rm Aut} ({\mathfrak g}) $ on the space $(H^2({\mathfrak g}, {\mathbb K})^m$ be a finite set. Then there is at least one rigid set of cocycles $\tilde c= (\tilde c_1,\tilde c_2,\dots,\tilde c_m)$.
\end{propos}

We begin the study of examples from the simplest case.
Every non-abelian three-dimensional nilpotent Lie algebra is metabelian and can be obtained as a one-dimensional central extension of a two-dimensional abelian algebra  ${\mathfrak m}_0(1)=\langle e_1, e_2\rangle$. Its cocycle $e^1\wedge e^2$ spans the intire space $H^*({\mathfrak m}(1))$.
Automorphism group  ${\rm Aut}({\mathfrak m}(1))=GL_2$ acts on the line $H^2({\mathfrak m}(1))=\langle e^1 \wedge e^2\rangle$ as multiplication by the determinant$\det A, A \in GL_2$ pf the matrix $A$ of the corresponding automorphism. The orbits of such an action will be only two: 1) single-point, consisting of the zero cohomology class; 2) an open orbit consisting of a complement to zero on the number line. Thus the cocycle $e^1\wedge e^2$ for the Lie algebra ${\mathfrak m}_0(1)$ is geometrically rigid and corresponding central extension ${\mathfrak m}_0(2)$ commonly called the three-dimensional Heisenberg Lie algebra
 ${\mathfrak h}_3$.The latter is isomorphic to the Lie algebra of strictly upper triangular matrices of order three and can be defined using the basis $e_1, e_2, e_3 $ and one non-trivial commutation relation $[e_1,e_2]=e_3$ (the remaining commutation relations have the form $ [e_i,e_j]=0 $). As a methodical corollary, we have obtained the well-known classification of three-dimensional nilpotent Lie algebras, up to isomorphism, there are only two: 1) an abelian Lie algebra and 2) a three-dimensional Heisenberg Lie algebra
 ${\mathfrak h}_3$.

We can now continue the process of central extensions and consider the extension of the Heisenberg algebra ${\mathfrak h}_3$. 
\begin{example} 
\label{Heisenberg}
Take as an extendable Lie algebra ${\mathfrak g} $ the three-dimensional Heisenberg Lie algebra ${\mathfrak h}_3$. We will consider its one-dimensional central extensions, i.e. $m=1$. The space of the second cohomology $H^2({\mathfrak m}_0(2),{\mathbb K})$ is two-dimensional (it is spanned by the cocycles  $e^1\wedge e^3$ and $e^2 \wedge e^3$) and, by removing the zero cohomology class from it and taking its quotient by the action $GL_1$, we obtain the projective line ${\mathbb K}P^1$. The group of automorphisms ${\rm Aut}({\mathfrak m}_0(2))$ acts on $\Lambda^*({\mathfrak m}_0(2)^*)$  as follows
$$
\begin{array}{c}
\varphi^*(e^1)=a_{11}e^1+a_{21}e^2, \varphi^*(e^2)=a_{12}e^1+a_{22}e^2,\\

\varphi^*(e^3)=\det Ae^3 +\beta_3e^2+\alpha_3e^1,\\
\varphi^*([e^1\wedge e^3])=\det A(a_{11}[e^1\wedge e^3]+a_{21}[e^2\wedge e^3]);\\
\varphi^*([e^2\wedge e^3])=\det A(a_{12}[e^1\wedge e^3]+a_{22}[e^2\wedge e^3]),
\end{array}
$$
where
$\det A=(a_{11}a_{22}-a_{21}a_{12})\ne 0$. Thus, the action ${\rm Aut}({\mathfrak m}_0 (2))$
on the projective line ${\mathbb K}P^1$ is equivalent to the standard action $GL_2$, which is transitive on ${\mathbb K}P^1$ (for two arbitrary straight lines on a plane passing through the origin, there is always a non-degenerate linear transformation that takes one straight line to another). Thus, an action   $GL_1\times{\rm Aut}({\mathfrak m}_0(2))$ on a space $H^2({\mathfrak m}_0(2), {\mathbb K})$ has only one non-trivial orbit  ${\mathcal O}$ which will be open.  Therefore, any nonzero cocycle from $H^2({\mathfrak m}_0(2),{\mathbb K})$ is rigid, and the corresponding one-dimensional central extension is isomorphic to the Lie algebra ${\mathfrak m}_0(3)$.
\end{example}
As a consequence of the results of the previous example, we obtain a classification of four-dimensional nilpotent Lie algebras, which does not depend on the choice of the ground field ${\mathbb K}$. Up to isomorphism, there are exactly three such Lie algebras: 1) four-dimensional abelian Lie algebra
 ${\mathbb K}\oplus {\mathbb K}\oplus {\mathbb K}\oplus {\mathbb K}$, 2) the direct sum of Lie algebras 
${\mathfrak h}_3\oplus {\mathbb K}$ and 3) ${\mathfrak m}_0(3)$.  As we noted at the end of the previous section, it would be interesting and useful to apply our method in higher dimensions.

Further, we deal with nilpotent Lie algebras that do not have rigid cocycles. According to the sentence \ref{finite_rigid} and Morozov’s classification \cite{Mor} dimension in which we can meet such a Lie algebra cannot be less than six.
\begin{example}
Consider $6$-dimensional graded Lie algebra ${\mathfrak m}_2(5)$ defined by its basis $e_1,e_2,e_3,e_4,e_5,e_6$ and structure relations
\begin{equation}
\label{m_2}
\begin{array}{c}
[e_1,e_2]=e_3, [e_1,e_3]=e_4,  [e_1,e_4]=e_5,  [e_1,e_5]=e_6,  \\

[e_2,e_3]=e_5,  [e_2,e_4]=e_6. \\
\end{array}
\end{equation}

We will study one-dimensional central extensions, i.e. $m=1$ in our general method.

The first thing to do is calculate
the second cohomology is $H^2({\mathfrak m}_2 (5), {\mathbb K})$. They are easily calculated using a second grading, which will be called a weight.
  Recall that the weight $w(\omega)$ of a monomial   $\omega=e^{i_1}\wedge e^{i_2}\wedge \dots \wedge e^{i_p}$ is equal to the sum of the superscripts of the factors in the monomial $\omega$.

The vector space $H^2({\mathfrak m}_2(5), {\mathbb K})$ is the linear span of the following cocycles 
$$
\Omega_7=[e^1{\wedge}e^6+e^2{\wedge}e^5], \omega_7=[e^2{\wedge}e^5-e^3{\wedge}e^4], \omega_5=[e^2{\wedge}e^3].
$$
Next, we find the action $\varphi^*$ of an arbitrary automorphism  $\varphi$
of the Lie algebra ${\mathfrak m}_2(5)$ on its cochain complex
\begin{equation}
\label{autom_m_2}
\begin{array}{c}
\varphi^*(e^1)=\alpha e^1, \varphi^*(e^2)=\alpha^2 e^2,
\varphi^*(e^3)=\alpha^3 e^3 +\beta_3e^2+\alpha_3e^1,\\
\varphi^*(e^4)=\alpha^4 e^4 +\alpha\beta_3 e^3+\beta_4e^2+\alpha_4e^1,\\
\varphi^*(e^5)=\alpha^5 e^5 +\alpha^2\beta_3e^4+ (\alpha\beta_4-\alpha_3\alpha^2)e^3+\beta_5e^2+\alpha_5e^1,\\
\varphi^*(e^6)=\alpha^6 e^6 +\alpha^3\beta_3e^5+ (\alpha^2\beta_4-\alpha_3\alpha^3)e^4+(\alpha\beta_5-\alpha_4\alpha^2)e^3
+\beta_6e^2+\alpha_6e^1.
\end{array}
\end{equation} 
Indeed, the conjugate action $\varphi^*$ of an automorphism $\varphi $ commutes with the differential $d$, which means that the closed forms $e^1$ and $e^2$ must go to closed forms
$$
\varphi^*(e^1)=a_{11}e^1+a_{21}e^2, \varphi^*(e^2)=a_{12}e^1+a_{22}e^2. 
$$ 
A structure relation 
$d\varphi^*(e^3)=\varphi^*(de^3)=\varphi^*(e^1)\wedge \varphi^*(e^2)$
gives the following value of $\varphi^*(e^3)$ 
$$
\varphi^*(e^3)=(a_{11}a_{22}-a_{12}a_{21}) e^3 +\beta_3e^2+\alpha_3e^1,
$$
for some constants $\beta_3,\alpha_3$.
The equality $a_{21}=0$ follows from exactness of the $2$-form on the right-hand side of the relation $d\varphi^*(e^4)=\varphi(e^1)\wedge \varphi(e^3)$.  The exactness of the form $\varphi(e^1)\wedge \varphi(e^4)+\varphi(e^2)\wedge \varphi(e^3)$ will imply the equality $a_{22}=a_{11}^2$ (we will further denote $a_{11}=\alpha$).  
The structure relation  $d\varphi^*(e^6)=\varphi^*d(e^6)$ implies $a_{12}=0$.

Taking quotient of the three-dimensional space $H^2({\mathfrak m}_2 (5), {\mathbb K})$ with punctured the zero cohomology class with respect to the action of the group $GL_1={\mathbb K}^*$ we get 
projective plane ${\mathbb K}P^2$.
The automorphism group $Aut({\mathfrak m}_2(5))$ acts on it. 
According to the formulas (\ref{autom_m_2}), we see that for an arbitrary
automorphism $\varphi$  from $Aut({\mathfrak m}_2(5))$ the following  equalities hold
$$
\varphi^*(\Omega_7)=\alpha^7\Omega_7,
\varphi^*(\omega_7)=\alpha^7\omega_7+(2\alpha^3\beta_4-\alpha\beta_3^2)\omega_5,
\varphi^*(\omega_5)=\alpha^5\omega_5.
$$
Fix cocycles  $\Omega_7, \omega_7, \omega_5$ as the basis of the space  $H^2({\mathfrak m}_2(5), {\mathbb K})$ and denote the affine coordinates corresponding to it as $x_1, x_2, x_3$. 
In homogeneous coordinates $(x_1: x_2: x_3)$ of the projective plane ${\mathbb K}P^2$ the automorphism action $\varphi$ from the group $Aut({\mathfrak m}_2(5))$ on ${\mathbb K}P^2$ is written as follows
 \begin{equation}
\label{proj_action}
(x_1:x_2:x_3) \to \left(x_1:x_2:\left(2\frac{\beta_4}{\alpha^4}-\frac{\beta_3^2}{\alpha^6}\right)x_2+\frac{1}{\alpha^2}x_3 \right),
\end{equation}
where parameters $\beta_4$, $\alpha$ and $\beta_3$ correspond to the action $\varphi$
on the cochain complex according to the formulas (\ref{autom_m_2}).
Thus, in the real case, the orbits of the group action (\ref{proj_action}) of $Aut({\mathfrak m}_2(5))$ on the projective plane ${\mathbb R}P^2$ are

1) point (at infinity) $(1: 0: 0)$;

2) the line $ tx_1-x_2 = 0, t \ ne 0, $ with a punctured point $ (0: 0: 1) $;

3) the point $ (0: 0: 1)$;

4) the half-line $ \{(x: 0: 1), x> 0 \} $;

5) the half-line $ \{(x: 0: 1), x <0 \} $.

Select the representatives in the orbits found above (on a straight line of the form $tx_1-x_2=0$ when $ t \ne 0 $ we take as its representative a point at infinity
 $(1:t:0)$)
$$
\left\{(1:t:0), t \in {\mathbb R}\right\}, (0:0:1), (1:0:1),(-1:0:1).
$$
In the case of a complex field, the last two points from our list will be in the same orbit represented by the point $(1: 0: 1)$. There will be no other changes resulting the transition field ${\mathbb K}={\mathbb C}$.

Obviously, in our example there will be no geometrically rigid cocycle.
Indeed, an infinite number of orbits of the form $tx_1-x_2=0$ passes through an arbitrary neighborhood $U$ of the three-dimensional space $H^2({\mathfrak m} _2(5), {\mathbb K})$ with coordinates $x_1, x_2, x_3$.

The algebras ${\mathfrak g}_ {7,t}$ obtained as central extensions of the algebra
${\mathfrak m}_2(5)$ using cocycles of the one-parameter family
$\Omega_t=\Omega_7+t{\omega}_7$ and relating to different values of the 
parameter $t$ will be pairwise non-isomorphic according to the theorem \ref{main2}. All of them will be filiform and positively graded.
The corresponding one-parameter family of $7$-dimensional Lie algebras
${\mathfrak g}_{7,t}$ is well known in the literature \cite{GJMKh, Mill2}.

We also note that the cocycle $x_1 \Omega_7 + x_2\omega_7 + x_3\omega_5 $ has filtration $6$ with respect to the natural filtration (which corresponds to filtration ${\mathfrak m}_2 (5)$ by ideals of the lower central series) if and only if $x_1 \ne 0 $.  

\end{example}

\section{Classification of filiform naturally graded Lie algebras}\label{s6}
The method of successive central extensions was applied in \cite{Mill5} for the classification of naturally graded Lie algebras
 $\mathfrak{g}=\oplus_{i=1}^{n}\mathfrak{g}_{i}$ satisfying the relations
\begin{equation}
\label{3/2}
\dim {\mathfrak g}_i+ \dim {\mathfrak g}_{i+1} \le 3, i=1,\dots, n-1.
\end{equation}

A complete classification of such algebras was obtained in \cite{Mill5}. The classification list looks rather cumbersome and we refer for details to \cite{Mill5}. In this paragraph we want:
1) illustrate the possibilities of the method of successive central extensions and present a new proof of the Vergne theorem on naturally graded filiform Lie algebras and 2) show why the classification problem for naturally graded Lie algebras $\mathfrak{g}=\oplus_{i=1}^{n}\mathfrak {g}_{i}$ of width two, i.e.
$\dim{\mathfrak g}_i \le 2, i = 1, \dots, n,$ is much more complicated than the problem of classifying Lie algebras satisfying conditions (\ref {3/2}).

Before proving the Vergne Theorem, we define one necessary subgroup in the group ${\rm Aut}(\mathfrak {g})$ of all automorphisms of the naturally graded Lie algebra ${\mathfrak g}$.
\begin{definition}
An automorphism $\varphi$ of a naturally graded Lie algebra
${\mathfrak g}=\oplus_{i=1}^n{\mathfrak g}_i$ is called graded automorphism, if all homogeneous subspaces  ${\mathfrak g}_i, i=1,\dots,n,$ are invariant with respect to $\varphi$.
$$
\varphi({\mathfrak g}_i)={\mathfrak g}_i, i=1,\dots,n,
$$
\end{definition}
Note that the arbitrary automorphism $\varphi: {\mathfrak g} \to {\mathfrak g}$ of the naturally only filtration based on its grading
$$
\varphi(\oplus_{i=k}^n{\mathfrak g}_i)=\varphi({\mathfrak g}^k)=\oplus_{i=k}^n{\mathfrak g}_i,\; k=1,\dots,n.
$$
We denote the subgroup of graded automorphisms of a
naturally graded Lie algebra
${\mathfrak g}=\oplus_{i=1}^n{\mathfrak g}_i$ by the symbol ${\rm Aut}_{gr}(\mathfrak{g})$.

The classification of central extensions of arbitrary nilpotent Lie algebras was based on two key theorems \ref{main1} and \ref{main2}. The classification of naturally graded Lie algebras lies in the "graded version" of these two theorems. We give the corresponding formulations, referring to the details of their evidence to \cite{Mill5}. 
\begin{theorem}
Let ${\mathfrak g}=\oplus_{i=1}^n{\mathfrak g}_i$ be a naturally graded Lie algebra with nil-index $s-1$ and let  $\tilde c_1,\dots,\tilde c_m$ be a set of cocycles  from 
$H^2_{(s)}({\mathfrak g}, {\mathbb K})$. The Lie algebra  $\tilde {\mathfrak g}$, defined as central extension which corresponds to $\tilde c_1,\dots,\tilde c_m$  is naturally graded Lie algebra of nil-index  $s$ and $\dim{ \tilde {\mathfrak g}_s}=m$ if and only if the cocycles  $\tilde c_1,\dots,\tilde c_m$ of grading $s$ are linearly independent.
\end{theorem}
\begin{theorem}
Let $\{ \tilde c_1,{\dots}, \tilde c_m \}$ and $\{ \tilde c_1',{\dots}, \tilde c_m'\}$ -- be two sets of cocycles with natural grading  $s$  in $H^2 _{(s)}({\mathfrak g}, {\mathbb K})$. 
They define isomorphic central extensions  $\tilde {\mathfrak g}$ and  $\tilde {\mathfrak g}'$ if and only if the linear spans  $\langle \tilde c_1,{\dots}, \tilde c_m \rangle$ are $\langle \tilde c_1',{\dots}, \tilde c_m'\rangle$ in the same orbit of linear action of the automorphism group  ${\rm Aut}_{gr}({\mathfrak g})$  on the subspace  $H^2_{(s)}({\mathfrak g}, {\mathbb K})$.   
\end{theorem}
As we see, in these theorems we speak of a natural graduation with respect to filtration by the ideals of the lower central series and reduce the group ${\rm Aut} ({\mathfrak g})$ of all automorphisms to its subgroup 
${\rm Aut}_{gr}({\mathfrak g})$. 
\begin{theorem}[Vergne \cite{V}]
\label{V_ne1}
Let $\mathfrak{g}=\oplus_{i=1}^{n{-}1}\mathfrak{g}_{i}$ be a naturally graded filiform Lie algebra. Then  

1) if $n=2k+1$ then $\mathfrak{g} \cong \mathfrak{m}_0(2k)$; 

2) if $n=4$ then $\mathfrak{g}$ is isomorphic to 
$\mathfrak{m}_0(3)$;

3) if $n=2k \ge 6$ then the Lie algebra $\mathfrak{g}$ is isomorphic either to
$\mathfrak{m}_0(2k-1)$ or to the Lie algebra  $\mathfrak{m}_1(2k-1)$ 
defined by its basis
$e_1, \dots, e_{2k}$ 
and structure relations
$$
[e_1, e_i ]=e_{i+1}, \; i=2, \dots, 2k{-}1; \quad \quad
[e_j, e_{2k{+}1{-}j} ]=(-1)^{j{+}1}e_{2k}, \quad j=2, \dots, k.
$$
\end{theorem}
\begin{proof}
We present a new proof of the Vergne theorem using the method of successive central extensions proposed in \cite{Mill2, Mill4}, later also used in \cite{BKT} to prove the conjecture of one of the authors of this article related to the Fialowski classification \cite{Fial1} of graded Lie algebras generated by two elements. The application of the method of successive central extensions in this work is based on one elementary observation.
$
\tilde{\mathfrak{g}}=\oplus_{i{=}1}^k{\mathfrak g}_i, \; {\mathfrak g}_k \ne 0.
$ 
According to the definition of a filiform Lie algebra, the dimension $\dim {\mathfrak g}=k+1$ and hence  $\dim {\mathfrak g}_1=2$ and $\dim {\mathfrak g}_ i=1$ for $2\le i \le k$.
Its last one-dimensional homogeneous addend ${\mathfrak g}_k$ coincides with the one-dimensional center (this follows from filiform property)
${\mathfrak g}_k=Z(\tilde {\mathfrak g})$.

Consider the quotient-algebra ${\mathfrak g}=\tilde{\mathfrak {g}}/{\mathfrak g}_k$. It is easy to see that it is also a naturally graded filiform Lie algebra and, as a vector space, coincides with
the direct sum of ${\mathfrak g}=\oplus_{i{=}1}^{k{-}1}{\mathfrak g}_i$.
Thus, the central extension is defined.
$$
0 \to{\mathfrak g}_k \to  \tilde{\mathfrak g}  \to {\mathfrak g} \to 0, 
$$
which corresponds to some cocycle $\tilde c$  in 
$H^2({\mathfrak g}, {\mathbb K})$. Important remark: the two-dimensional cocycle $\tilde c$ is not arbitrary, it has weight $k$ (in natural grading): $\tilde c \in H^2_{(k)}({\mathfrak g}, {\mathbb K})$.

To prove the Vergne theorem will be induction on the dimension of Lie algebras.
The basis of induction: the Lie algebra $\mathfrak{m}_0 (1)$ is a two-dimensional abelian Lie algebra with a base $e_1,e_2$ of two elements, each of which has a weight equal to one. Its unique cocycle is the form $e^1\wedge e^2$ of weight two. The one-dimensional central extension constructed from this cocycle defines a three-dimensional naturally graded filiform Lie algebra $\mathfrak {m}_0 (2)$. The next step was already analyzed in the example \ref{Heisenberg}. The corresponding central extension gave the Lie algebra $\mathfrak{m}_0 (3)$.

Let the theorem be proved for algebras of all
dimensions of $\le n$. Case a) $ n = 2m $. According to the inductive hypothesis in dimension
$2m$, up to isomorphism, is exactly two naturally graded filiform Lie algebras: $ \mathfrak {m}_0(2m-1)$ and $\mathfrak{m}_1(2m-1)$.
We find in each of these Lie algebras $2$-cocycles of weight $2m$.
In the algebra $\mathfrak {m}_0 (2m-1)$, up to multiplication by a scalar, there will be a unique cocycle of weight $2m + 1$ this is $e^1\wedge e^{2m}$ (recall that the $1$-form $e^{2m}$ has weight $2m-1$ in natural grading). In the Lie algebra $\mathfrak{m}_1(2m-1)$ cocycles of weight $2m$ simply does not exist. In this sense, the Lie algebra $\mathfrak {m}_1(2m-1)$ can not be extended to a naturally graded filiform Lie algebra of higher dimension.

The case b) $n=2m+1$. According to the inductive hypothesis, we have exactly one $(2m+1)$-dimensional naturally graded filiform Lie algebra and it is the Lie algebra $\mathfrak {m}_0(2m)$.
Its cohomology subspace $H^2_{(2m+1)}(\mathfrak {m}_0(2m), {\mathbb K})$ of weight $2m+1$ will be two-dimensional. Its basis, for instance, can be chosen as follows: $e^1\wedge e^{2m+1}$ and $\sum_{i=2}^m(-1)^ie^i \wedge e^{2m + 3-i} $. Recall that exact forms of weight $2m+1$ do not exist. What is the structure of  the automorphism group ${\rm Aut}(\mathfrak{m}_0(2m))$? It is easy to verify the following formulas
for the action $\varphi^*$ of an arbitrary automorphism
$\varphi: {\mathfrak m}_0(n) \to  {\mathfrak m}_0(n)$ on the dual space  $\mathfrak{m}_0(2m)^*$ for $n \ge 4$ ($\alpha \ne 0, \mu \ne 0, \beta \in {\mathbb K}$):
\begin{equation}
\label{m_0_automorphism}
\begin{array}{c}
\varphi^*(e^{1})=\alpha e^{1},
\varphi^*(e^2)=\beta e^{1} +\mu e^2,\\
\varphi^*(e^{j})=\alpha^{j{-}1} \mu e^{j}+\sum_{i=j+1}^n \gamma_{j,i}e^i,  2 \le j \le n-1. 
\end{array}
\end{equation}
We are interested in automorphisms that preserve the invariant subspace of $2$-forms of weight $2m+1$. Consider the subgroup ${\rm Aut}_ {gr}(\mathfrak {m}_0(2m))$ automorphisms of the form (\ref{m_0_automorphism}) with all $ \gamma_{j, i}=0$. The matrix $A_ {\varphi^*}$ of the action of such an automorphism $\varphi$ on the invariant subspace $H^2_{(2m + 1)}(\mathfrak {m} _0(2m), {\mathbb K}) $ will be triangular
 $$
\alpha^{2m{-}1}\mu\begin{pmatrix} \alpha & \beta \\ 0 & \mu\end{pmatrix}, 
\; \alpha, \mu \ne 0.
$$
Now we recall that we still have the group $GL_1={\mathbb K}^*$ on the cohomology subspace $H^2_{(2m+1)}(\mathfrak {m}_0(2m), {\mathbb K })$ weights $2m + 1$. The actions $GL_1$ and ${\rm Aut}_ {gr}(\mathfrak{m}_0(2m)) $ commute and we can first take the quotient of
the subspace $H^2_{(2m+1)}(\mathfrak {m}_0(2m), {\mathbb K})$ by the action $ GL_1$, and then consider the action ${\rm Aut}_{gr}(\mathfrak {m}_0(2m))$ on the corresponding quotient
${\mathbb P}H^2_{(2m+1)}(\mathfrak {m}_0(2m), {\mathbb K})= {\mathbb K}P^1$.
It is easy to see that the corresponding projective action
has exactly two orbits represented by the points $ (1: 0) $ and $ (1: 1) $. These points of the projective line correspond to one-dimensional central extensions $ {\mathfrak m}_0 (2m{+}1)$ and $ {\mathfrak m}_1(2m{+}1) $, respectively.
\end{proof}

Let us analyze our proof. What helped us in the calculations? Answer: 1) the small dimension of the subspaces $H^2_{(k)}(\mathfrak {g}, {\mathbb K}) $ and 2) the small dimension of the automorphism subgroup ${\rm Aut}_{gr}(\mathfrak {g})$. These two circumstances led to the fact that at each step we have the orbit space of the action $GL_1\times {\rm Aut}_{gr} (\mathfrak {g})$ on the homogeneous subspace $ H^2_{(k)}(\mathfrak {g}, {\mathbb K})$ was not just finite, but consisted of no more than two orbits.

In the papers \cite{Mill1, Mill2}, it was shown that the classification problem for $ \mathbb N$-graded filiform Lie algebras ${\mathfrak g}=\oplus_{i=1}^k{\mathfrak g} _i$, which all homogeneous components of $ {\mathfrak g}_i $ are one-dimensional, can be solved by an inductive process of successive one-dimensional central extensions. Later, the same idea was applied to the classification of another class of
$\mathbb N $-graded Lie algebras ${\mathfrak g} = {\mathfrak g}_1\oplus {\mathfrak g}_3 \oplus {\mathfrak g}_4 \oplus \dots$ with one lacuna ${\mathfrak g}_2 = 0 $ in the grading \cite{BKT}.

\section{Elementary orbital geometry of a group action $Aut({\mathfrak g})$ on the Grassmannian $Gr (m,H^2({\mathfrak g}, {\mathbb K}))$}\label{s7}

In this section we explain why the classification problem for naturally graded Lie algebras $\mathfrak {g}=\oplus_{i = 1}^{n}\mathfrak {g}_{i}$ of width two, i.e.
$\dim{\mathfrak g} _i \le 2, i = 1, \dots, n$, is fundamentally more difficult than the analogous problem for naturally graded Lie algebras of "width $ {\frac{3}{2}}$" from \cite{Mill5}. Namely, we show that already in small dimensions there exist parametric families of pairwise nonisomorphic naturally graded Lie algebras of width two. On the other hand, the study of the orbits of a torus on the $Gr(2, {\mathbb C}^4)$ Grassmannian is a very nontrivial problem, as was shown in \cite {BuTer}.

To build the required example, first consider
the ${\mathcal L}(m,n)={\mathcal L}(m)/{\mathcal L}(m)^{n+1}$ quotient Lie algebra of the free Lie algebra ${\mathcal L}(m)$ from $m$ generators with respect to its ideal ${\mathcal L}(m)^{n+1}$ of
lower central series. Such a Lie algebra is often called the free nilpotent Lie algebra of $m$ generators of  degree of nilpotency $n$. Obviously, ${\mathcal L}(m, n)$ is a naturally graded nilpotent Lie algebra.

\begin{lemma}[\cite{Mill5}]
\label{vazhnaya_lemma}
Consider the free nilpotent Lie algebra ${\mathcal L}(2,3)$ of two generators of the nilpotency degree $3$.
The orbit space of the action $GL (1, {\mathbb K}) \times {\rm Aut} ({\mathcal L} (2,3))$ on the three-dimensional space $H^2({\mathcal L}(2,3), {\mathbb K}) $ consists of:

a) of four orbits, two of which are open, in the case of a real field ${\mathbb K} = {\mathbb R}$;

b) of three orbits, one of which is open, in the complex case
${\mathbb K}={\mathbb C}$.
\end{lemma}

\begin{proof}
The cochain complex of the Lie algebra ${\mathcal L}(2,3)$ is given by generators 
$a^1, b^1, a^2, a^3, b^3,$ and formulas for the differential $d$
\begin{equation}
\begin{split}
da^1=db^1=0, \;\;  da^2=a^1\wedge b^1, \\
 da^3=a^1 \wedge a^2, \;\;  db^3=b^1 \wedge a^2;
\end{split}
\end{equation}
The automorphism group ${\rm Aut}({\mathcal L}(2,3))$ acts on the generators $a^1, b^1, a^2, a^3, b^3$ of the $\Lambda^*({\mathcal L}(2,3))$  according to the formulas
\begin{equation}
\label{automorphism_2_3}
\begin{array}{c}
\varphi^*(a^1)=\alpha_1 a^1 + \rho_1 b^1, \varphi^*(b^1)=\beta_1 a^1 + \mu_1 b^1,\\
\varphi^*(a^2)=(\alpha_1\mu_1-\beta_1\rho_1) a^2+\alpha_2 a^1+\rho_2 b^1,\\
 \varphi^*(a^3)=(\alpha_1\mu_1-\beta_1\rho_1)( \alpha_1a^3+\rho_1b^3)+(\alpha_1\rho_2-\alpha_2\rho_1)a^2+\alpha_3 a^1+\rho_3 b^1,\\
 \varphi^*(b^3)=(\alpha_1\mu_1-\beta_1\rho_1)( \beta_1a^3+\mu_1b^3)+(\alpha_1\mu_2-\beta_2\rho_1)a^2+\beta_3 a^1+\mu_3 b^1. 
\end{array}.
\end{equation}
Choose the following basis of three-dimensional space $H^2({\mathcal L}(2,3),{\mathbb K})$:
$$
a^1\wedge a^3, \; a^1\wedge b^3+b^1\wedge a^3, \; b^1\wedge b^3.
$$
Fixing the corresponding coordinates in the space $H^2({\mathcal L}(2,3),{\mathbb K})$, we get explicit formulas
for the action $\varphi^*$ for an arbitrary $\varphi \in {\rm Aut}({\mathcal L}(2,3))$.
\begin{equation}
\label{action_L_23}
\varphi^*=
(\alpha_1\mu_1-\beta_1\rho_1)\cdot 
\begin{pmatrix}
\alpha_1^2 & 2\rho_1\alpha_1 &\rho_1^2\\
\alpha_1\beta_1 & \rho_1\beta_1{+}\alpha_1 \mu_1& \mu_1\rho_1\\
\beta_1^2 & 2\mu_1\beta_1& \mu_1^2
\end{pmatrix}.
\end{equation}

The single point orbit of the zero cohomology class will correspond to the trivial central extension ${\mathcal L}(2,3)\oplus {\mathbb K}$.
\end{proof}

Taking the quotient of the set of nonzero cohomology classes from
$H^2({\mathcal L}(2,3),{\mathbb K})$ by the action $GL_1$ we get
projective softness of ${\mathbb K}P^2$, on which the group acts
  ${\rm Aut} ({\mathcal L}(2,3))$ by the formula \ref{action_L_23}. In the real case, the orbit of the point $ (0: 0: 1) $ is the oval $ \{(\rho^2{:} \mu\rho {:} \mu^2) \} \subset {\mathbb R}P^2$, which is represented by the parabola $y^2=x$ in the standard affine map $x_3 \ne 0 $ with coordinates $x= \frac{x_1}{x_3}, y=\frac {x_2}{x_3}$.

\begin{tikzpicture}
\draw[->] (-2.5,0)--(3,0) node[anchor=north] {$x$}; 
\draw[->] (0,-2)--(0,2.5) node[anchor=east] {$y$}; 
\draw [rotate=-90][very thick] (-1.5,2.25) parabola bend (0,0)(1.5,2.25);
       [rotate=-90] (-1.5,2.25) parabola bend (0,0)(1.5,2.25);
\put(16,-0.9){$\bullet$};
\put(-17.5,-0.9){$\bullet$};
\put(-0.9,-0.9){$\bullet$};
\put(12.5,2){$(1{:}0{:}1)$};
\put(-10,-6){$(0{:}0{:}1)$};
\put(-22,2){$({-}1{:}0{:}1)$};
\put(14,20){$x{=}y^2$};
\put(30,10){$\{(\rho^2{:}\mu\rho{:}\mu^2)\} \subset {\mathbb R}P^2$};
\put(45,5){$\mu \neq 0$};
\end{tikzpicture}

The action of the group ${\rm Aut} ({\mathcal L}(2,3))$ (actually the action of $ GL_2$) on the projective plane ${\mathbb R}P^2$ has two more orbits, they coincide with the inner and outer regions of the parabola $y^2=x$ on the affine map $x_3 \ne 0$. These two orbits can be represented by the points
$({-}1:0:1)$ and $(1:0:1)$. Remark that if ${\mathbb K}={\mathbb  C}$ then $({-}1:0:1)$ and $(1:0:1)$ will be in the same orbit of the action ${\rm Aut} ({\mathcal L}(2,3))$
$$
({-}1:0:1) \in \{ (\alpha^2+\rho^2, \alpha\beta+\mu\rho, \beta^2+\mu^2)\}, \;  \mu=1, \beta=\rho=0, \alpha^2= {-}1.
$$
Consider now a new Lie algebra $\tilde {\mathcal L}(2,4)$. 
\begin{definition}
The naturally graded Lie algebra $\tilde {\mathcal L}(2,4)$ is defined by the basis and commutation relations
$$
\begin{array}{c}
\tilde {\mathcal L}(2,4)=\langle a_1, b_1\rangle \oplus \langle a_2 \rangle
\oplus \langle a_3, b_3\rangle \oplus \langle a_4, b_4\rangle,\\

[a_1,b_1]=a_2, [a_1,a_2]=a_3, [b_1,a_2]=b_3, [a_1,a_3]=a_4, [b_1,b_3]=b_4, [a_1,b_3]=[b_1,a_3]=0.
\end{array}
$$
\end{definition}
The Lie algebra $\tilde{\mathcal L}(2,4)$ is a two-dimensional central extension of the free nilpotent algebra ${\mathcal L}(2,3) $ given by a set of two cocycles
 $$
(c_1, c_2)=(a^1\wedge a^3, b^1\wedge b^3).
$$
The elements of the basis $a_i, b_i$ of the Lie algebra $ \tilde{\mathcal L}(2,4)$ are
homogeneous elements of weight $i$: the subscript of each basic element coincides with its weight.

The cochain complex $\Lambda^*(\tilde {\mathcal L}(2,4)^*)$ can be defined  by generators $a^1, b^1, a^2, a^3, b^3, a^4, b^4$  and structure formulas for the differential $d$
\begin{equation}
\begin{split}
da^1=db^1=0, \;\;  da^2=a^1\wedge b^1, \\
 da^3=a^1\wedge a^2, \;\;  db^3=b^1\wedge a^2,\\
 da^4=a^1\wedge a^3, \;\;  db^4=b^1\wedge b^3,
\end{split}
\end{equation}
\begin{lemma}
The two-dimensional cohomology of $H^2(\tilde {\mathcal L}(2,4),{\mathbb K})$ is the direct sum of two nontrivial homogeneous subspaces
$$
H^2(\tilde {\mathcal L}(2,4),{\mathbb K})=H^2_{(5)}(\tilde {\mathcal L}(2,4),{\mathbb K})\oplus H^2_{(4)}(\tilde {\mathcal L}(2,4),{\mathbb K})
$$
The subspace $H^2_{(5)}(\tilde {\mathcal L}(2,4),{\mathbb K})$ of weight $5$ is four-dimensional and can be defined as the linear span of the following basic cocycles
\begin{equation}
\label{cohomology_basis}
a^1\wedge a^4, \; a^1\wedge b^4+a^2\wedge b^3, \; b^1\wedge a^4-a^2\wedge a^3, \;b^1\wedge b^4.
\end{equation} 
And the one-dimensional homogeneous subspace $H^2_{(4)}(\tilde {\mathcal L}(2,4),{\mathbb K})$ is spanned by the cocycle of weight $4$
$$ 
a^1\wedge b^3+b^1\wedge a^3.
$$ 
A subgroup of graded automorphisms ${\rm Aut}_{gr}({\mathcal L}(2,4))\subset {\rm Aut}({\mathcal L}(2,4))$ is isomorphic to the group of non-degenerate lower triangular matrices $LT(2,{\mathbb R})$, 
whose action with respect to the basis (\ref{cohomology_basis})
for  $H^2_{(5)}(\tilde {\mathcal L}(2,4),{\mathbb K})$  is written as follows
\begin{equation}
\label{3Daction}
A=
\begin{pmatrix}
\alpha & 0\\
\rho & \mu
\end{pmatrix} \in LT(2,{\mathbb K}) \to \alpha^2\mu\cdot 
\begin{pmatrix}
\alpha^2 & 3\rho\alpha &\alpha\rho & 2\rho^2\\
0&\alpha\mu & 0 & \mu\rho\\
0&0 & \alpha\mu& \mu\rho\\
0&0&0& \mu^2
\end{pmatrix}, a^1 \to \alpha a^1+\rho b^1, b^1 \to \mu b^1. 
\end{equation}
\end{lemma}
\begin{proof}
We compute the cohomology $H^2(\tilde {\mathcal L}(2,4),{\mathbb K})$ using the second (natural) grading. Obviously, there are no non-trivial cocycles in weights strictly less than $4$. The algebra $\tilde {\mathcal L}(2,4)$ will have only one non-trivial cocycle of weight $4$, this is $a^1\wedge b^3+b^1\wedge a^3$. 
The basis of cocycles of weight $5$ is found directly by writing down the action of the differential $d$ on the basis of a subspace of cochains of weight five
$$
a^1\wedge a^4, a^1\wedge b^4, b^1\wedge a^4,  b^1\wedge b^4, a^2\wedge a^3, a^2 \wedge b^3.
$$
Let us prove the formula (\ref{3Daction}). For the action of a graded automorphism $\varphi$ from ${\rm Aut}_{gr}({\mathcal L}(2,4))$ on generators $a^1, b^1$, $a^2$, $a^3, b^3$, $a^4, b^4$ of a cochain complex $\Lambda^*( \tilde {\mathcal L}(2,4))$ the formulas (\ref{automorphism_2_3}) are valid, where
$$
\alpha_2=\rho_2=\alpha_3=\rho_3=\beta_3=\mu_3=0.
$$
These formulas can be extended to the automorphism action on generators $a^4, b^4$ if and only if $\beta_1=0$. Redefining the automorphism parameters $\alpha_1 = \alpha, \mu_1 = \mu, \rho_1 = \rho $, and explicitly calculating its action on the basis (\ref {cohomology_basis}) from cocycles of weight $5$, we get the formulas (\ref{3Daction}). 
\end{proof}
Hence we have the group $GL_1\times LT(2,{\mathbb R})$ acting on the four dimensional space $H^2_{(5)}(\tilde {\mathcal L}(2,4),{\mathbb K})$ according to the formulas  (\ref{3Daction}).  Remove the one-point orbit of the zero cohomology class and go to the projectivization in homogeneous coordinates
 $(x_1:x_2:x_3:x_4)$.  

Consider the action
in the affine chart $x_4 \ne 0$, where we fixed the affine coordinates $x=\frac{x_1}{x_4},y=\frac{x_2}{x_4},z=\frac{x_3}{x_4}$. We introduce the parameters $a=\frac{\alpha}{\mu}, b=\frac{\rho}{\mu}$ for the action of the automorphism $\varphi$ on the space
$PH^2_{(5)}(\tilde {\mathcal L}(2,4),{\mathbb K})$. Rewrite the formulas (\ref {3Daction}) in affine coordinates $x, y, z$ of the chart $x_4 \ne 0$:
$$
\begin{pmatrix}
x\\
y\\
z
\end{pmatrix} \to
\begin{pmatrix}
a^2 & 3ab &ab \\
0&a & 0\\
0&0& a
\end{pmatrix} 
\begin{pmatrix}
x\\
y\\
z
\end{pmatrix} 
+
\begin{pmatrix}
2b^2\\
b\\
b
\end{pmatrix}, a, b \in {\mathbb K}, a\ne 0.
$$
Such a $LT(2,{\mathbb K})$-action  on the three-dimensional space ${\mathbb K}^3$ has an invariant plane
$y= z$. The orbit of the origin of $ (0,0,0) $ will be a parabola
 $\left\{ \begin{array}{c}x=2y^2,\\ y=z\end{array}\right.$. 
Further, the analysis of the orbit space of the $LT(2,{\mathbb K})$-action on the plane $y=z$ depends on which ground field ${\mathbb K} $ we consider (the situation here is completely similar to the case of the free nilpotent Lie algebra $ {\mathcal L} (2,3) $): two open orbits represented by points
$(\pm1,0,0)$ in the real case and one orbit for ${\mathbb K} = {\mathbb C}$, represented by the point $(1,0,0)$.

Consider an arbitrary point $P_0=(x_0,y_0,z_0), y_0 \ne z_0,$ in the affine space ${\mathbb K}^3$, not lying in the plane $y = z$. There is a unique automorphism
$\varphi \in  {\rm Aut}_{gr}(\tilde {\mathcal L}(2,3))$ sending $P_0$ to the point $
\varphi(P_0)$ lying on the line $x=t, y=0, z=1$. Coordinates of such a point
could be found explicitely
$$
\varphi(P_0)=\left(\frac{x_0-y_0^2-y_0z_0}{(z_0-y_0)^2} ,0,1\right).
$$
Thus, the orbit ${\mathcal O}(P_0) $ of an arbitrary point $P_0$ not lying in the plane $ y = z $ intersects the line $x=t, y=0, z=1$ at a single point $\varphi(P_0)$. Further, for clarity, we will consider only the real case ${\mathbb K}={\mathbb R}$.

\begin{figure}[h]
\center{\includegraphics[scale=0.38]{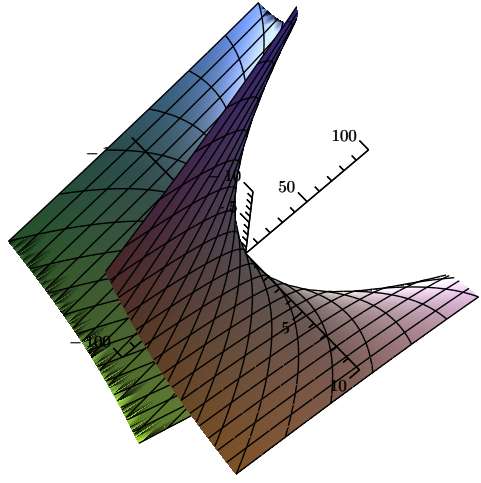}}
\caption{Hyperbolic paraboloid $t(z-y)^2+y^2+yz-x=0$ for $t< \frac{1}{8}$.}
\label{hyperbolic_hyperboloid}
\end{figure}

The orbit ${\mathcal O}(P_0)$ of an arbitrary point $P_0=(x_0, y_0, z_0), y_0 \ne z_0, $ is given by the parametric equation in $ {\mathbb R}^3$
$$
r(a,b)=(a^2x_0+3aby_0+abz_0+2b^2,ay_0+b,az_0+b), a, b \in {\mathbb R}, a \ne 0.
$$
This is a second order surface equation $\Gamma_{x_0,y_0,z_0}$ in the space ${\mathbb R}^3 $. If you enter the parameter $t=\frac{(x_0-y_0 ^ 2-y_0z_0)}{(z_0-y_0)^2}$, then the same surface $\Gamma_{x_0, y_0, z_0}$ can be set even more simple implicit second order equation $F_t (x, y, z) = 0 $, depending on the real parameter $t$
\begin{equation}
\label{paraboloids}
F_t(x,y,z)=t(z-y)^2+y^2+yz-x=0.
\end{equation}
What about the structure of this one-parameter family of second-order surfaces? Each such surface $F_t (x, y, z) =0$ intersects
plane $y = z$ at parabola  $\left\{ \begin{array}{c}x=2y^2,\\ y=z\end{array}\right.$.
This parabola should be removed from the surface $F_t(x, y, z) = 0$ to get the  orbit ${\mathcal O} _t $ of the point $(t,0,1)$.

It is easy to see that for $t_0=\frac{1}{8}$ the surface $F_{t_0}(x,y,z)=0$ will be a parabolic cylinder. In all other cases, $F_{t}(x,y,z)=0$ will be a paraboloid: for $t <\frac{1}{8} $ it is hyperbolic paraboloid, and for $ t> \frac{1}{8}$ it is elliptic one. Fixing the value of $ R_0> 0 $, we see that the intersection points of the ball $\left\{(x,y,z), x^2+y^2+z^2 \le R_0^2\right\}$ s with the elliptic paraboloid $F_t(x,y,z)=0$, tend to $t\to+\infty$ to the inner part of the parabola $ \left \{\begin {array}{c} x = 2y^2, \\ y = z \end{array} \right.$ on the plane $y = z$. Accordingly, as $ t \to - \infty $, the points $(x,y,z), x^2+y^2+z^2 \le R_0^2$, belonging to the hyperbolic paraboloid $F_t(x,y,z)=0$, tend to the outer part of the parabola $\left\{ \begin{array}{c}x=2y^2,\\ y=z\end{array}\right.$ of the plane $ y = z $. Moreover, each paraboloid twofold covers the corresponding region of the plane $y = z$.

\begin{figure}[h]
\center{\includegraphics[scale=0.38]{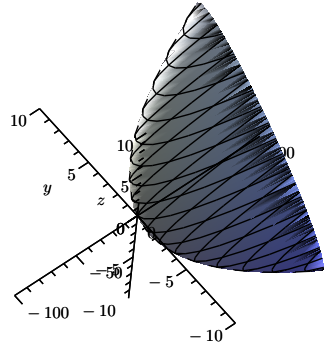}}
\caption{Elliptic paraboloid $t(z-y)^2+y^2+yz-x=0$ for $t> \frac{1}{8}$.}
\label{elliptic_hyperboloid}
\end{figure}

Thus, all our space ${\mathbb R}^3$, except for the plane $y = z$, is fibered into paraboloids, and through any neighborhood $U$ of an arbitrary point $P$ of the space ${\mathbb R}^3$ passes an infinite set of paraboloids (orbits of action).

It remains to investigate the orbits on an invariant with respect to the action of the ${\rm Aut}_{gr}(\tilde {\mathcal L}(2,3))$ group of the $x_4=0$ plane of the projective space
${\mathbb P}H^2_{(5)}(\tilde {\mathcal L}(2,3),{\mathbb  R})$.
The action of the element$\varphi \in {\rm Aut}_{gr}({\mathcal L}(2,3))$, restricted to this plane, is written as
$$
\varphi(x_1:x_2:x_3:0)=(ax_1+3bx_2+bx_3:x_2:x_3:0), 
$$
where $a=\frac{\alpha}{\mu}, b=\frac{\rho}{\mu}$.

It is easy to see that its orbit space will consist of:

1) single-point orbit ${\mathcal O} = \{(1: 0: 0: 0) \} $ (the common asymtotic direction of all elliptic paraboloids from the family (\ref{paraboloids}));

2) projective lines of the form $Ax_2 + Bx_3 = 0$ 
(asymptotic planes of hyperbolic paraboloids from the family (\ref{paraboloids})).

These orbits can be set by their representatives
$$
(1:0:0:0), (0:1:0:0), \{(0:\tau:1:0), \tau \in {\mathbb R}\}.
$$
Representatives of the previously found orbits from the $x_4 \ne 0$ map should be added to them, respectively:
$$
(0: 0: 0: 1), (\pm1: 0: 1), \{(t: 0: 1: 1), t \in {\mathbb R} \}.
$$
\begin{corollary}
There are two one-parameter families ${\mathcal L}_t$ and $\tilde{\mathcal L}_ {\tau} $ of  real pairwise non-isomorphic naturally graded Lie algebras of width two:
$$
\begin{array}{c}
{\mathcal L}_t=\langle  a_1, b_1\rangle \oplus \langle  a_2 \rangle
\oplus \langle  a_3, b_3\rangle \oplus \langle a_4, b_4\rangle \oplus \langle a_5\rangle,\\

[a_1,b_1]=a_2, [a_1,a_2]=a_3, [b_1,a_2]=b_3, [a_1,a_3]=a_4,\\

 [b_1,b_3]=b_4,
[a_1,b_3]=[b_1,a_3]=0,\\

[a_1,a_4]= ta_5, [a_1,b_4]=0, [b_1,a_4]=a_5, \\

[a_2,a_3]=-a_5, [b_1,b_4]=a_5, [a_2,b_3]=-a_5,
\end{array}
$$
and also
$$
\begin{array}{c}
\tilde {\mathcal L}_{\tau}=\langle \tilde a_1,\tilde b_1\rangle \oplus \langle \tilde a_2 \rangle
\oplus \langle \tilde a_3,\tilde b_3\rangle \oplus \langle \tilde a_4, \tilde b_4\rangle \oplus \langle \tilde a_5\rangle,\\

[ \tilde a_1, \tilde b_1]= \tilde a_2, [ \tilde a_1, \tilde a_2]= \tilde a_3, [\tilde b_1, \tilde a_2]= \tilde b_3, \\

[\tilde a_1, \tilde a_3]= \tilde a_4, [ \tilde b_1, \tilde b_3]= \tilde b_4,
[ \tilde a_1, \tilde b_3]=[ \tilde b_1, \tilde a_3]=0,\\

[ \tilde a_1, \tilde a_4]= 0, [ \tilde a_1, \tilde b_4]=\tau \tilde a_5, [ \tilde b_1, \tilde a_4]= \tau \tilde a_5, [\tilde b_1, \tilde a_4]=a_5, \\

[ \tilde a_2, \tilde a_3]=- \tilde a_5, [ \tilde b_1, \tilde b_4]= 0, [ \tilde a_2, \tilde b_3]=0.
\end{array}
$$
\end{corollary}


\end{document}